\documentclass[aop,MSNbibl,seceqn,dvips]{arximspdf}

\doi{10.1214/12-AOP806} 
\volume{41}
\issue{3B}
\pubyear{2013}
\firstpage{1864}
\lastpage{1899}

\makeatletter

\renewcommand{\mid}{|}

\newcommand{\cal}{\mathcal}

\newtheorem{theorem}{Theorem}[section]
\newtheorem{claim}[theorem]{Claim}

\newtheorem{lemma}[theorem]{Lemma}
\newtheorem{corollary}[theorem]{Corollary}
\newtheorem{proposition}[theorem]{Proposition}

\newproclaim{definition}[theorem]{Definition}

\newcommand{\E}{\mathbb{E}}
\renewcommand{\Pr}{\mathbb{P}}
\newcommand{\Bin}{\operatorname{Bin}}
\newcommand{\Var}{\operatorname{Var}}
\newcommand{\N}{\mathbb{N}}
\newcommand{\R}{\mathbb{R}}

\newcommand{\MinSum}{\mathcal{W}_{\mathrm{MS}}}
\newcommand{\Expl}{\mathcal{W}_{\mathrm{EX}}}

\def\inlaw{\stackrel{\cal L}{=}}

\newcommand{\pathweight}{\operatorname{pw}}

\setattribute{abstract}{width}{290pt}

\makeatother

\begin{document}
\begin{frontmatter}

\title{On explosions in heavy-tailed branching random~walks}
\runtitle{Explosions in heavy-tailed branching random walks}

\begin{aug}
\author[A]{\fnms{Omid} \snm{Amini}\ead[label=e1]{oamini@math.ens.fr}},
\author[B]{\fnms{Luc} \snm{Devroye}\ead[label=e2]{luc@cs.mcgill.ca}},
\author[C]{\fnms{Simon} \snm{Griffiths}\corref{}\ead[label=e3]{sgriff@impa.br}}
\and
\author[D]{\fnms{Neil} \snm{Olver}\thanksref{m1}\ead[label=e4]{olver@math.mit.edu}}
\runauthor{Amini, Devroye, Griffiths and Olver}
\affiliation{\'Ecole Normale Sup\'erieure, McGill University, IMPA and MIT}
\address[A]{O. Amini\\
CNRS---DMA\\
\'Ecole Normale Sup\'erieure\\
45 Rue d'Ulm, 75230F Paris\\
France\\
\printead{e1}}
\address[B]{L. Devroye\\
School of Computer Science\hspace*{47.8pt}\\
McGill University\\
3480 University Street\\
Montr\'eal, Qu\'ebec, H3A 2A7\\
Canada\\
\printead{e2}}
\address[C]{S. Griffiths\\
IMPA\\
Est. Dona Castorina 110\\
Jardim Bot\^anico\\
Rio de Janeiro\\
Brazil\\
\printead{e3}}
\address[D]{N. Olver\\
Department of Mathematics\\
Massachusetts Institute of Technology\\
Cambridge, Massachusetts 02139-4397\\
USA\\
\printead{e4}} 
\end{aug}

\thankstext{m1}{Supported in part by NSF Grant CCF1115849.}

\received{\smonth{2} \syear{2011}}
\revised{\smonth{10} \syear{2012}}

%
\begin{abstract}
Consider a branching random walk on $\R$, with offspring distribution
$Z$ and nonnegative displacement distribution $W$. We say that
\textit{explosion} occurs if an infinite number of particles may be found
within a finite distance of the origin. In this paper, we investigate
this phenomenon when the offspring distribution $Z$ is heavy-tailed.
Under an appropriate condition, we are able to characterize the pairs
$(Z, W)$ for which explosion occurs, by demonstrating the equivalence
of explosion with a seemingly much weaker event: that the sum over
generations of the minimum displacement in each generation is finite.
Furthermore, we demonstrate that our condition on the tail is best
possible for this equivalence to occur.

We also investigate, under additional smoothness assumptions, the
behavior of $M_n$, the position of the particle in generation $n$
closest to the origin, when explosion does not occur (and hence
$\lim_{n \rightarrow\infty} M_n = \infty$).\looseness=-1
\end{abstract}

%

\end{frontmatter}

\section{Introduction}\label{intro}

Our aim in this paper is to give a classification of the displacement
random variables
in heavy-tailed branching random walks in $\mathbb R$ for which
explosion---a concept we will define shortly---occurs.
Thus, consider a branching random walk on $\R$.
The process begins with a single particle at the origin; 
this particle moves to another point of $\R$ according to a
displacement distribution $W$, where it gives birth to a random number
of offspring, according to a distribution $Z$.
This procedure is then repeated: the particles in a given generation
each take a single step according to an independent copy of the same
distribution~$W$, and then give birth to the next generation.
We consider the case where $W$ is nonnegative (in which case the
process is also called an \textit{age-dependent} process; the
displacement of a particle can also be interpreted as a birthdate).
Let $\Gamma_t$ be the number of particles with displacement at most
$t$; then we say that \textit{explosion} occurs if $\Gamma_t = \infty$
for some finite $t$.

Alternatively, let $M_n$ be the displacement of the leftmost particle
in the $n$th generation.
If the process dies out and there are no particles remaining in the
$n$th generation, then define $M_n = \infty$.
Explosion is the event that $\lim_{n \rightarrow\infty} M_n <\infty$.
Note that, since $M_n$ is monotone, it has a limit.

Taking a tree view of the above process, denote by $T_{Z}$ a random
Galton--Watson tree
with offspring distribution $Z$, and let $Z_n$ be the number of
children at level $n$.
To avoid the trivial case, we assume throughout that $\Pr\{ Z= 1 \} < 1$.
Each edge of $T_{Z}$ is then independently given a weight according to
the nonnegative distribution~$W$.
The connection to the above process is that the displacement of a node
is simply the sum of the weights on the path from the root to that node.
From this perspective, which is the one we will take in this paper,
explosion is the event that there exists an infinite path for which the
sum of the weights on the path is finite.

In the process of studying the event of explosion, we first consider
the case where
the offspring distribution has finite mean. The different cases
described in the next paragraph
show that we can either trivially solve the problem or reduce to the
most interesting case of an infinite mean.

\subsection*{Reduction to the case of an infinite mean}

Consider a Galton--Watson process with offspring distribution $Z$
satisfying $0 < \E\{ Z\} < \infty$. We still assume $\Pr\{ Z= 1 \} < 1$.
Let $W$ be a weight (or displacement) distribution on the edges of the
Galton--Watson tree.

Consider first the case where $\Pr\{ W = 0 \} = 1$. In this case,
explosion is equivalent to the event
that the Galton--Watson tree is infinite, that is, the survival of the
Galton--Watson process.
In that case, if $\E\{ Z\} \le1$, there is no survival,
and if $\E\{ Z\} > 1$, there is a positive probability of
survival~\cite{AN72}. From now on we will assume that $\Pr\{ W = 0 \}
< 1$ and assume that the Galton--Watson process is supercritical.

In the case of a supercritical Galton--Watson process, under the
assumption $\E\{Z\} <\infty$, the results of Hammersley~\cite{Ham74},
Kingman~\cite{Kin75} and Biggins~\cite{Big90} show the existence of a
constant $\gamma$ such that, conditional on the nonextinction of the
process, $M_n /n$ tends to $\gamma$ almost surely. This shows that the
random variables~$M_n$, conditional on survival, behave linearly in
$n$, that is, $M_n = \gamma n+o(n)$. One consequence of the
Hammersley--Kingman--Biggins theorem is that if $\gamma> 0$, then
explosion never happens. Now define
\[
H:= \E\{Z\}\Pr\{W = 0\}.
\]
It can be shown 
that $\gamma= 0$ if and only if $H\geq1$.
We consider in fact three cases: $H< 1$, $H> 1$ and $H= 1$.\vspace*{9pt}

$\bullet$ \textsc{Case I}: $H< 1$.\quad
Here, as stated above, explosion occurs with probability zero.
This can be seen more simply as follows: fix an $\varepsilon> 0$ such
that $\Pr\{W < \varepsilon\} < (\E\{Z\})^{-1}$ and mark all\vadjust{\goodbreak} edges with
weight smaller than $\varepsilon$.
Then each component in the forest of marked edges is a subcritical
Galton--Watson tree, and hence has finite size almost surely. Thus, any
infinite path must contain an infinite number of unmarked edges, and
hence cannot be an exploding path.\vspace*{9pt}

$\bullet$ \textsc{Case II}: $H> 1$.\quad
In this case, explosion happens with probability one. To see this, take
a sub-Galton--Watson tree by keeping only children for which \mbox{$W = 0$}.
This tree is supercritical and thus survives
with some positive probability $\rho$. It follows that with positive
probability,
there is an infinite path of length zero.
Since, conditional on survival, explosion is a $0$--$1$ event (for a
proof see later in this \hyperref[intro]{Introduction}),
we infer that it happens with probability one. A theorem of Dekking and
Host~\cite{DeHo91} ensures the
existence of an almost surely finite random variable $M$ such that
$M_n$ converges a.s. to $M$. Under the extra condition $\E{Z^2}
<\infty$, they determine stronger results on the
limit distribution $M$.\vspace*{9pt}

$\bullet$ \textsc{Case III}: $H= 1$.\quad
This threshold case is the most intriguing---it was already considered
in an earlier pioneering work of Bramson~\cite{Bra78} and in the work
of Dekking and Host~\cite{DeHo91}. In this case, the occurrence of
explosion is a delicately balanced event that depends upon the behavior
of the distribution of $W$ near the origin and on the distribution of $Z$.

Bramson's main theorem is the following result on the behavior of $M_n$
under the assumption that there exists a $\delta>0$ such that $\E\{
Z^{2+\delta}\} <\infty$.
For any fixed $\lambda$, define $\sigma_{\lambda,n} = p +
(1-p)e^{-\lambda^n}$ where $p=\Pr\{W=0\}<1$.
Then explosion happens if and only if there exists some $\lambda>1$
such that
$\sum_{n=1}^\infty F_W^{-1}(\sigma_{\lambda,n}) <\infty$. In the
case of no explosion, and
conditional on the survival of the branching process, the following
convergence result on the asymptotic of
$M_n$ holds. Almost surely, we have
%
\begin{equation}
\label{bramconv} \lim_{n\rightarrow\infty} \frac{M_n}{\sum_{k=1}^{s(n)}
F_W^{-1}(\sigma_{2,k})} =1,
\end{equation}
where $s(n) = \lceil\log\log n/\log2 \rceil$.
We refer to~\cite{DeHo91} for a generalization of Bramson's theorem to
the case of $\E\{Z^2\}<\infty$, under some extra mild
conditions.\vspace*{9pt}

Following Bramson~\cite{Bra78}, we first transform the tree $T_Z$ into
a new tree $T'$ as follows.
The roots are identical.
First consider the sub-Galton--Watson tree rooted at the root of $T_Z$
consisting only of
children (edges) that have zero weight.
This subtree is critical.
For any distribution of $Z$ satisfying the threshold condition, note
that the size $S$ of the
sub-Galton--Watson tree is a random variable $S \ge1$ with $\E\{ S \}
= \infty$.
In some cases, we know more---for example, when $\Var\{ Z\} = \sigma^2
\in(0, \infty)$,
then $\Pr\{ S \ge k \} \sim\sqrt{2/ \pi\sigma^2 k}$ as $k \to
\infty$ (see, e.g., the book of Kolchin~\cite{Kol86}).
All of the nodes in $S$ are mapped to the root of the new tree $T'$.
The children of that root in $T'$ are all the children of the mapped
nodes in $T_Z$ that did not have $W=0$.

Let $X_i$ be the number of vertices of degree $i$ in the
sub-Galton--Watson tree.
The number of children of the root of $T_Z$ is distributed as
\[
\zeta= \sum_{i=0}^\infty\sum
_{j=1}^{X_i} \zeta_{i,j},
\]
where $\zeta_{i,1},\zeta_{i,2}, \ldots$ are i.i.d. random variables
having distribution of a random variable
$\zeta_i$. In addition, the distribution of $\zeta_i$ is given by
\[
\Pr\{\zeta_i = k\} = c_i \pmatrix{k+i \cr i} \bigl(1- \Pr
\{W=0\}\bigr)^k \Pr\{ W=0\}^i \Pr\{Z = k+i\},
\]
where $c_i$ is a normalizing constant. Note that $\sum_{i\geq0} X_i =
S$.

For each child of the root in $T'$, repeat the above collapsing
procedure.
It is easily seen that $T'$ itself is a Galton--Watson tree with
offspring distribution~$\zeta$.
The moment generating function $G_\zeta(s)$ of $\zeta$ is easily seen
to satisfy the functional equation
%
\begin{equation}
\label{eqfunctional} G_\zeta(s) = G_Z \bigl(\bigl(1- \Pr
\{W=0\}\bigr)s + \Pr\{W=0\}G_\zeta(s) \bigr).
\end{equation}
Furthermore, the displacement distribution is $W$ conditional
on $W > 0$.
Finally, one can verify that $\E\{ \zeta\} = \infty$.
More importantly, explosion occurs in $T_Z$
if and only if explosion happens in $T'$.
We have thus reduced the explosion question to one for
a new tree in which the expected number of children is infinite
and in which $W$ does not have an atom at zero.

Observe that the transformation described in case III
is valid whenever $W$ has an atom at the origin.
In particular, this construction can also be used to eliminate
an atom at the origin when $\Pr\{ W = 0 \} > 0$ and $\E\{ Z\} =
\infty$.
In this case, we still have $\E\{ \zeta\} = \infty$.

It follows from the above discussion that in the study of the event of
explosion, we need to consider only the (most interesting) case where
\[
\E\{ Z\} = \infty,\qquad \Pr\{ W = 0 \} = 0.
\]
All our results below are concerned only with this case.

\subsection*{A simple necessary condition for explosion}
There is a rather obvious necessary condition for explosion.
Let $Y_i$ be the minimum weight edge at level $i$ in the tree.
Then the sum of weights along any infinite path is certainly at least
$\sum_{i=1}^\infty Y_i$.
We say that a fixed weighted tree is \textit{min-summable} if this sum
is bounded; if a tree is not min-summable, it cannot have an exploding path.

For any fixed, infinite, rooted tree $T$, and distribution $W$ on the
nonnegative reals, let $T^W$ denote a random weighted tree obtained by
weighting each edge with an independent copy of $W$.
For a fixed tree $T$ and weight distribution $W$, it follows easily
from Kolmogorov's $0$--$1$ law that explosion and min-summability of
$T^W$ are both $0$--$1$ events.
Thus, we make the following definitions.
%
\begin{definition}
For any infinite rooted tree $T$:
\begin{longlist}
\item let $\Expl(T)$ be the set of weight distributions so that $T^W$
contains an exploding path almost surely, and

\item let $\MinSum(T)$ be the set of weight distributions so that
$T^W$ is min-summable almost surely.
\end{longlist}
\end{definition}

In this new notation, the observation above is simply that $\Expl(T)
\subseteq\MinSum(T)$, for any tree $T$.
Unsurprisingly, in general, $\Expl(T)$ may be strictly contained
within $\MinSum(T)$.
For example, consider an infinite binary tree $T$ and a uniform weight
distribution $W$ on $[0,1]$.
Except with probability at most $\exp(-2^{i/2})$, the minimum of
$2^{i}$ copies of $W$ is at most $2^{-i/2}$.
Thus, with positive probability $\sum_{i\ge1}Y_{i}\le\sum_{i\ge
1}2^{-i/2}<3$, and so $W\in\MinSum(Z)$.
On the other hand, we may easily prove that $W\notin\Expl(Z)$, that is,
that the probability that there exists an exploding path is zero. To
see this, consider the event $A_{i}$ that
there exists a path from the root to level $i$ of weight less than $i/128$.
The existence of an exploding path certainly implies that for all
sufficiently large $i$, $A_i$ occurs.
We now observe that $\Pr\{A_{i}\}\le2^{-i}$. Indeed, the event
$A_{i}$ implies that there is a path\vspace*{1pt} from the
root to level $i$ at least half of whose edges have weight less than
$\frac1{64}$.
Since there are only $2^{i}$ paths to level $i$ and at most $2^{i}$
ways to choose a subset of the edges of a fixed path,
and since for each path and each fixed subset of at least $\frac i2$
edges the probability that all these edges
have weight less than $\frac1{64}$ is at most $8^{-i}$, the bound
easily follows. The same proof shows that
for the exponential distribution $E$, no explosion can happen [however,
$E \in\MinSum(T)$; this follows from example
(iv) of Section~\ref{secminsummability}].

\subsection*{Main results}
It may appear that, aside from some trivial cases, $\MinSum(T)$ should
always strictly contain $\Expl(T)$.
However, somewhat counterintuitively, this is not the case;
there are examples of trees with generation sizes growing very fast
(double exponentially)
for which $\Expl(T)=\MinSum(T)$. Consider, for example, the tree $T$
defined as follows:
all nodes of generation $n$ have $2^{2^{n}}$ children. In this case,
for a given weight
distribution $W$, the distribution of the sum of minimum weights of
levels is
\[
\sum_{n\ge1} \min_{ 1\leq i \leq2^{(2^n-1)}}
W_n^i,
\]
where each $W_{n}^{i}$ is an independent copy of $W$.
Also, the path constructed by the simple greedy algorithm, which,
starting from root,
adds at each step the lowest weight edge from the current node to one
of its children,
has total weight distributed as
\[
\sum_{n\ge1} \min_{ 1\leq i \leq2^{2^{(n-1)}}}
W_n^i.
\]
The property of these sums being finite almost surely is clearly equivalent,
so that $\Expl(T)=\MinSum(T)$.
Our main result is that this phenomenon is in fact quite general
in trees obtained by a Galton--Watson process with a heavy-tailed
offspring distribution.
We call the distribution $Z$ \textit{plump} if for some positive
constant $\varepsilon$ the inequality
%
\begin{equation}
\label{eqcondition}
\Pr\bigl\{Z\ge m^{1+\varepsilon}\bigr\}\ge\frac{1}{m}
\end{equation}
holds for all $m$ sufficiently large.
Equivalently, $Z$ is plump if its distribution function $F_{Z}$
satisfies $F_Z^{-1}(1-1/m) \geq m^{1+\varepsilon}$ for $m$ sufficiently
large. We remark that $\E{Z}=\infty$ for any plump $Z$.
\begin{mainthm*}
Let $Z$ be a plump distribution.
Let $T$ be a random Galton--Watson tree with offspring distribution $Z$,
but conditioned on survival.
Then
\[
\Expl(T) = \MinSum(T) \qquad\mbox{with probability } 1.
\]
\end{mainthm*}

We now state a second form of the Equivalence theorem. For this, we
must extend the definition
of $\Expl$ and $\MinSum$ to Galton--Watson offspring distributions.
Let $Z$ be an offspring distribution and $W$ a weight distribution.
We have the following:
%
\begin{claim}
For a given offspring distribution $Z$ and weight distribution~$W$, and
conditioning on survival of the Galton--Watson process, explosion and
min-summability are $0$--$1$ events.
\end{claim}
\begin{pf}
Let $(W_i)_{i=1}^\infty$ be a sequence of independent copies of $W$,
let $(S_i)_{i=1}^\infty$ be a random walk with jump distribution given
by $Z-1$, and let $(X_i)_{i=1}^{\infty}$ be the increments.
In the usual way, this random walk can be thought of as representing
(in breadth-first fashion) a sequence of one or more Galton--Watson
trees, with $X_i+1$ giving the number of children at step $i$ and $W_i$
the weight of the $i$th edge.
Since $\E Z > 1$, one of these trees $T'$ will be infinite with
probability 1, and this tree is exactly a Galton--Watson tree
conditioned on survival.
The sequence $((X_i, W_i))_{i=1}^{\infty}$ clearly encodes all the
information about $T'$, and the two events under consideration are tail
events with respect to this sequence;
thus, Kolmogorov's $0$--$1$ law applies. The same argument holds for
min-summability.
\end{pf}
We can thus define $\Expl(Z)$ and $\MinSum(Z)$ for an offspring
distribution $Z$ as follows:
\[
\Expl(Z) := \bigl\{W | W \in\Expl(T_Z) \mbox{ almost surely
conditioned on survival} \bigr\}
\]
and
\[
\MinSum(Z) := \bigl\{W | W \in\MinSum(T_Z) \mbox{ almost surely
conditioned on survival} \bigr\}.\vadjust{\goodbreak}
\]
The alternative (though slightly weaker) formulation of the Equivalence
theorem can now be stated as follows:
\begin{mainthmalt*}
$\!\!$For a plump distribution~$Z$,
\[
\Expl(Z) = \MinSum(Z).
\]
\end{mainthmalt*}

Min-summability is clearly a simpler kind of condition than explosion;
in particular,
it depends only on the generation sizes $Z_n$ rather than the full
structure of the tree $T_Z$.
Indeed, the Equivalence theorem becomes more interesting if one
observes that it is possible to derive the following
quite explicit necessary and sufficient condition for min-summability.
%
\begin{theorem}\label{thmcomputableequivalence}
Given a plump offspring distribution $Z$,
let $m_0 > 1$ be large enough such that the condition (\ref
{eqcondition}) holds for all $m \geq m_0$.
Define the function $h\dvtx  \N\rightarrow\mathbb{R}^+$ as follows:
%
\begin{equation}\label{eqdeff}
h(0) = m_0 \quad\mbox{and}\quad h(n+1) = F_{Z}^{-1}
\bigl(1-1/h(n)\bigr) \qquad\mbox{for all } n \geq1.
\end{equation}
Then for any weight distribution $W$, $W \in\MinSum(Z)$ and, hence,
also $W \in\Expl(Z)$, if and only if
\[
\sum_n F_W^{-1}
\bigl(h(n)^{-1}\bigr) < \infty.
\]
\end{theorem}

Given the Equivalence theorem above, one may wonder if there is a way
to weaken the condition given in (\ref{eqcondition}) such that the
theorem still remains valid.
We show that this condition is to some extent the best we can ask for.
More precisely, we prove the following:
\begin{sharpness*} Let $g\dvtx \N\to\N$ be an increasing function satisfying
\[
g(m)=m^{1+o(1)}.
\]
Then there is an offspring distribution $Z$ satisfying $\Pr\{Z\geq
g(m)\}\ge1/m$ for all $m\in\mathbb N$, but for which $\Expl(Z)\neq
\MinSum(Z)$.
\end{sharpness*}

So far our results concerned the appearance of the event of explosion,
however, it is also natural to ask
how fast $M_n$ tends to infinity in the case there is a.s. no exploding path.
Although there is no reason to expect a convergence theorem in the case
of no explosion for general
plump distributions in the absence of any smoothness condition on the
tails of $Z$,
we show that a stronger plumpness property allows to obtain precise
information on the rate of convergence to infinity of~$M_n$.
To explain this, note that the plumpness assumption on $Z$ is
equivalent to $1-F_Z(k) \geq k^{-\eta}$\vadjust{\goodbreak}
for $\eta= \frac1{1+\varepsilon}$ and for all $k$ sufficiently large.
Consider now the stronger smoothness
condition
%
\begin{equation}
\label{eqsmooth} 1-F_Z(k) = k^{-\eta}\ell(k),
\end{equation}
where $\ell$ is any continuous and bounded function which is nonzero
at infinity.

\begin{limittheorem*} Let $Z$ satisfy the smoothness condition, and let
$W$ be any weight distribution with $W \notin\Expl(Z)$. Then a.s.
conditional on survival,
\[
\lim_{n\rightarrow\infty} \frac{M_n}{\sum_{k=1}^{n} F_W^{-1}
(\exp(-(1+\varepsilon)^k ) )}=1
\]
for all $\varepsilon> 0$.
\end{limittheorem*}

Applying a Tauberian theorem (see Section~\ref{secconvergence} for
more details), we find that
condition (\ref{eqsmooth}) is equivalent to
the condition
\[
K_Z(s):=1-G_Z(1-s) \sim a s^\eta\ell
\biggl(\frac1{s}\biggr)
\]
near $s=0$ for some $a>0$; recall $G_Z$ is the moment generating
function of~$Z$. Going back to case III of the finite mean case and the
transformation described there, we observe that the use of the
functional equation (\ref{eqfunctional}) allows to translate the
smoothness condition above, imposed on the modified offspring
distribution $\zeta$ of infinite mean (obtained after the
transformation), to a smoothness condition on~$Z$, the original
distribution of finite mean. In particular,
\[
K_{\zeta}(s)=1-G_\zeta(1-s) \sim a s^{1/(1+\varepsilon)} \bigl(1+O
\bigl(s^\beta\bigr) \bigr) \qquad\mbox{for $s$ near zero}
\]
for some $a, \varepsilon, \beta>0$ is equivalent to a condition of the form
%
\begin{equation}
\label{eqtails} K_{ Z}(s) \sim\E\{Z\} s -c s^{1+\varepsilon} \bigl(1 +
O\bigl(s^\delta\bigr)\bigr) \qquad\mbox{for $s$ near zero}
\end{equation}
for some $c,\delta>0$. We note that condition (\ref{eqtails})
assumes some regularity on the tails of $Z$ but the variance could be
infinite, thus, the above result can be regarded as a strengthening of
Bramson's theorem~\cite{Bra78}.

\subsection*{Further related work}

The literature on explosion
is partially surveyed by Vatutin and Zubkov~\cite{VZ93}. The early
work deals with exponentially distributed weights:
in this case, there is no explosion almost surely if and only if
\[
\sum_{n=1}^\infty{ 1 \over n \sum_{r=0}^n \Pr\{ Z > r \} } <
\infty
\]
(see~\cite{Har63}, Section V. 6,~\cite{Schu82,Do84}).
This condition cannot be simplified; Grey~\cite{Gr89} showed
that there does not exist any fixed function $\psi\ge0$ such that explosion
would be equivalent to $\E\{ \psi(Z) \} = \infty$.\vadjust{\goodbreak}

Some general properties of the event of explosion
were obtained in~\cite{Sev67} by considering the generating functions
of the number
of particles born before time $t$, parametrized by $t$, and looking at
the nonlinear
integral equation satisfied by these generating functions. By using this
analytic approach and under some smoothness conditions on the distribution
function $F_W$ of the displacement $W$, Sevast'yanov \cite
{Sev67,Sev70}, Gel'fond~\cite{Gel67}
and Vatutin~\cite{Vat76,Vat87} obtain necessary and sufficient
conditions on the event
of explosion. The result of Vatutin~\cite{Vat87} can be stated as
follows. Consider the case $\Pr\{ W=0 \} = 0$ and suppose that zero is
an accumulation point of $W$, that is,
the distribution function $F_W$ of $W$ satisfies $F_W (w) > 0$ for all
$w > 0$.
Assume the following regular variation style condition holds:
there exists $\lambda\in(0,1)$ such that
%
\begin{equation}
\label{eqregular} 0 < \liminf_{t \downarrow0} {F_W^{-1} (\lambda t)
\over F_W^{-1} (t)
} \le
\limsup_{t \downarrow0} {F_W^{-1} (\lambda t) \over F_W^{-1} (t) } < 1.
\end{equation}
Then explosion does not occur if and only if for all
$\varepsilon> 0$,
%
\begin{equation}
\label{eqintcrit} \int_0^\varepsilon{F_W^{-1}
\biggl(\frac{s}{K_Z(s)}} \biggr) \,\frac
{ds}{s} = \infty.
\end{equation}
Condition (\ref{eqregular}) basically forces $F_W$ to
behave in a polynomial manner near the origin.
Indeed, if $F_W(w) \sim w^\alpha$ for some $\alpha> 0$ as $w
\downarrow0$,
then $F_W^{-1} (t) \sim t^{1/\alpha}$
as $t \downarrow0$, and so (\ref{eqregular}) holds. The exponential
law corresponds to $\alpha=1$, for example.
The criterion given by (\ref{eqintcrit}) was earlier proved to be
necessary and sufficient
for nonexplosion by Sevast'yanov~\cite{Sev67,Sev70} and
Gel'fond~\cite{Gel67}
under the slightly more
restrictive condition that $F_W (w)/w^\alpha\in[a,b]$ for all~$w$,
where $0 < a \le b < \infty$ and $\alpha\ge0$.
As soon as we leave that polynomial oasis, Vatutin's condition is violated.
Examples include $F_W (w) \sim\exp(-1/w^\alpha)$
and $F_W(w) \sim1 / \log^\alpha(1/w)$ for $\alpha> 0$.

A quite general sufficient (but not necessary) condition without any explicit
regularity assumption on $W$ was proved by Vatutin~\cite{Vat96} for
explosion in nonhomogenous branching
random walks.
In the homogenous case, the result states that if there exists a
sequence of nonnegative reals $(y_n)_{n\in\mathbb N}$
such that $\lim_n y_n = 0$ and
\[
\sum_{n=1}^\infty F_W^{-1}
\bigl(y_n/K_{Z_n}(y_n) \bigr) < \infty,
\]
then explosion occurs. This result is close in spirit to our
Equivalence theorem, but we stress that the results are distinct---we
see no way in which one may be deduced from the other.

More precise information on the behavior and convergence to infinity of
$M_n$ can be obtained in the finite mean case and under extra
conditions. Recall that in the finite mean case, $M_n = \gamma n +
o(n)$ for some $\gamma\geq0$.
McDiarmid showed in~\cite{McD95} that $M_n -\gamma n = O(\log n)$ if
$\E\{Z^2\} <\infty$ and $W$ has an exponential upper\vadjust{\goodbreak} tail. Recently,
Hu and Shi~\cite{HuSh09}
proved that if the displacements are bounded and \mbox{$\E\{Z^{1+\varepsilon
}\} <\infty$} for any $\varepsilon>0$,
then, conditional on survival, $(M_n-\gamma n)/\log n$ converges in
probability but, interestingly,
not almost surely. (We note in passing that this work and the recent
work of A\"idekon and Shi~\cite{AiSh11} provide
Seneta--Heyde norming results~\cite{BiKy97} in the boundary case.)
Under the extra assumption that $Z$ is bounded,
Addario-Berry and Reed~\cite{AdRe09} calculate $\E\{M_n\}$ to within
$O(1)$ and
prove exponential tail bounds for $\Pr\{|M_n-\E\{M_n\}| > x\}$.
Extending these results, A\"{i}dekon~\cite{Ai11} proves the
convergence of $M_n$ centered around its median for a large
class of branching random walks. For tightness results in general,
under some extra assumptions on
the decay of the tail distribution or weight distribution, see
Bachmann~\cite{Bac00}
and Bramson and Zeitouni~\cite{BrZe09,BrZe07}.

\subsection*{Organization of the paper}
Section~\ref{secprel} will concern some preliminaries, mostly
involving what we call the
\textit{speed} of an offspring distribution. In Section~\ref{secproof},
we prove the Equivalence theorem. The proof is somewhat algorithmic in nature
and shows that a certain (infinite) algorithm will always find an
exploding path
under the given conditions. In Section~\ref{secminsummability}, we
prove Theorem~\ref{thmcomputableequivalence} and give some examples
calculating the condition for specific cases.
In Section~\ref{seccounterexample} we provide a generic counterexample
that shows that the equivalence does not hold if we weaken the
conditions in any substantial way,
proving the sharpness of condition (\ref{eqcondition}). Finally, in
Section~\ref{secconvergence} we prove the limit theorem under
condition (\ref{eqsmooth}).

\section{Preliminaries}\label{secprel}
In this section we present some definitions and results needed for the
proof of the Equivalence theorem.
That theorem (in its second form) is concerned with the equivalence of
$\MinSum(Z)$ and $\Expl(Z)$ for certain offspring distributions $Z$.
Thus, it will be important to have a good characterization of whether a
weight distribution $W$ belongs to $\MinSum(Z)$,
in other words, whether $\sum_{n \geq1} \min\{{W}_n^1,\ldots,
{W}_n^{Z_n}\}$ is finite,
each $W_n^i$ being an independent copy of $W$.
To do this, we will introduce two notions.
The first is the concept of the \textit{speed} of a branching process,
from which we will obtain an understanding of the growth of the
generation sizes $Z_n$.
The second is the concept of \textit{summability with respect to an
integer sequence}, which concerns the
behavior of sums of the form $\sum_{n\ge1}\min\{{W}_n^1,\ldots,
{W}_n^{\sigma_n}\}$ for a given integer sequence $(\sigma_{n})_{n\in
\N}$.

\subsection*{Speed of a Galton--Watson branching process}

We introduce the concept immediately and then give a number of examples.
%
\begin{definition}\label{defspeed}
An increasing function $f\dvtx  \N\rightarrow\mathbb{R}^+$, taking only
strictly positive values, is called a \textit{speed} of a Galton--Watson
offspring distribution $Z$ if there exist positive integers $a$ and $b$
such that with positive probability
\[
Z_{n/a } \leq f(n) \leq Z_{ bn } \qquad\mbox{for all } n\in\N.
\]
(Here, we set $Z_x = Z_{\lfloor x \rfloor}$ for $x \in\R$.)\vadjust{\goodbreak}
\end{definition}
Note that there is a small issue of extinction here, and that is why we
insist that $f$ is strictly positive, otherwise $f(n)=0$ would be a
speed for any distribution with $\Pr\{Z= 0\} > 0$.

\subsection*{Examples of speeds} Here we give examples of speeds for
various distributions~$Z$:
\begin{longlist}
\item
If $\E\{Z\}\le1$, then almost surely $Z_{n}=0$ for all sufficiently
large $n$, and so $Z$ does not have a speed.
\item
If $\E\{Z\}=m\in(1,\infty)$, then Doob's limit law states that the
random variables $V_{n}=Z_{n}/m^{n}$ form a martingale sequence with
$\E{V_{n}}\equiv1$, and $V_{n}\to V$ almost surely, where $V$ is a
nonnegative random variable. Furthermore, in the case that $Z$ is
bounded, the limit random variable $V$ has mean $1$ (and so, in
particular, $\Pr\{V\ge1\}> 0$). From this we may easily verify that
$m^{n}$ is a speed of $Z$. Indeed, Doob's limit law implies that the
inequality $Z_n \le(M+1) m^n$
holds for all $n$ large enough, with probability at least $P(V \le M)$.
Taking $M$ sufficiently large, this probability may be made arbitrarily
close to $1$. For the lower bound, one may consider a truncation $Z'$
of $Z$ such that $\E\{Z'\}\ge\sqrt{m}$. Since $Z'$ is bounded, we
deduce that in the truncated branching process associated with $Z'$
there is a positive probability that $Z'_{n}\ge m^{n/2}/2$ for all
sufficiently large $n$. Since there is a natural coupling such that
$Z_{n}\ge Z'_{n}$ for all $n$, this completes our proof that $m^{n}$ is
a speed of~$Z$.
\item
If $Z$ is defined by $\Pr\{Z\ge m+1\}=m^{-\beta}$ for each $m\ge1$,
where $\beta\in(0,1)$, then $Z$ is plump [one may take $\varepsilon
=\beta^{-1}-1$ in condition (\ref{eqcondition})] and the double
exponential function $f(n)=2^{(\beta^{-1})^{n}}$ is a speed of $Z$.
Heuristically, this follows from the fact that, conditioned on the
value of $Z_{n}$, one would expect $Z_{n+1}$ to be of the order
$Z_{n}^{\beta^{-1}}$.
A formal proof follows from Theorem~\ref{thmspeed} together with the
observation that the function
$h$ appearing in that theorem is equivalent to $f$ as a speed
[i.e., there exist $a',b'\in\mathbb{N}$ such that the inequalities
$f(\lfloor n/a'\rfloor)\le h(n)\le f(b'n)$
hold for all $n$]. Indeed, as we will explain in Section~\ref{secconvergence},
a much stronger statement holds in this case.
\item
If $Z$ is defined by $\Pr\{Z\ge m\}=1/\log_{2}{m}$ for each $m\ge2$,
then $Z$ is plump. Applying Theorem~\ref{thmspeed}, we find that the
tower function $h(n)$ defined by $h(0)=2$ and $h(n+1)=2^{h(n)}$ for
$n\ge0$ is a speed of $Z$.
\end{longlist}

\subsection*{Summable weight distributions with respect to an integer sequence}

Let $W$ be a random variable with nonnegative values. Let $\sigma
=(\sigma_n)_{n \in\N}$
be a sequence of positive integers and $W_n^j$ be a family of
independent copies of $W$ for $n,j \in\N$.
Define the sequence of minima
\[
\Lambda_n:= \min_{1\leq j\leq\sigma_n} W_n^j.
\]
The random variable $W$ is called \textit{$\sigma$-summable} if there
is a positive probability
that $\sum_n \Lambda_n$ is finite.\vadjust{\goodbreak}

Note that the event in the above definition is a $0$--$1$ event. Thus,
if $W$ is $\sigma$-summable, then $\sum_n \min_{1 \leq j \leq\sigma
_n} W_n^j$ is finite with probability one. For a characterization of
$\sigma$-summable weight distributions see Proposition \ref
{propsigma-equivalence}. Examples are given at the end of Section \ref
{secminsummability}.

We note that if $W$ is $\sigma$-summable and $\tau$-summable, then
$W$ is $\sigma\cup\tau$-summable, and if $\sigma_n \leq\tau_n$
for all $n$, $\sigma$-summability implies $\tau$-summability.
We also have the following:
%
\begin{lemma}\label{lemminsumseq}
Let $\sigma$ be any increasing sequence, and let $\tau$ be defined by
$\tau_n = \sigma_{\gamma n}$ for some constant $\gamma$, a positive integer.
Then $W$ is $\sigma$-summable iff it is $\tau$-summable.
\end{lemma}
\begin{pf}
Write $\sigma= \sigma^0 \cup\sigma^1 \cup\cdots\cup\sigma^{\gamma-1}$,
where $\sigma^i:= \{ \sigma_{\gamma n+i}\dvtx  n \in\N\}$.
Since $\sigma$ is increasing, if $W$ is $\sigma^i$-summable and $i <
j$, then $W$ is $\sigma^j$-summable.
So if $W$ is $\tau=\sigma^0$-summable, then it is $\sigma^i$-summable
for all $0 \leq i\leq\gamma-1$, and thus $\sigma$-summable.
The other direction follows trivially since $\tau\subseteq\sigma$.
\end{pf}

The following proposition relates the condition of the Equivalence
theorem to the notion of $\sigma$-summability under the presence of a
speed function for the Galton--Watson distribution.
%
\begin{proposition}\label{propsigma-addable}
Let $W$ be a weight distribution and $Z$ an offspring distribution. %
Suppose that $f\dvtx \N\rightarrow\mathbb{R}^+$ is a speed for $Z$.
Then $W \in\MinSum(Z)$ if and only if $W$ is $\sigma$-summable for
the sequence $\sigma=(f(n))_{n\in\N}$.
\end{proposition}
\begin{pf}
Since $f$ is a speed for $Z$, the event
\[
R:= \bigl\{Z_{n/a} \leq f(n) \leq Z_{bn} \mbox{ for all }
n \bigr\}
\]
occurs with positive probability.
Let $\sigma^a$ be the sequence given by $\sigma^a_n = f(an)$, and
$\sigma^b$ the sequence defined by $\sigma^b_n = f(\lfloor n/b
\rfloor)$.
Suppose $W$ is $\sigma$-summable; then by Lemma~\ref{lemminsumseq},
$W$ is $\sigma^b$-summable.
Whenever $R$ occurs, $Z_n \geq\sigma^b_n$ for all $n$ and, hence,
$T_Z$ has the min-summability property almost surely.
Thus, $W \in\MinSum(T_Z)$ with positive probability, and hence $W \in
\MinSum(Z)$.

Conversely, if $W$ is not $\sigma$-summable, then again by Lemma \ref
{lemminsumseq},
it is not $\sigma^a$-summable.
Thus, even when conditioning on survival, $W \notin\MinSum(T_Z)$ with
positive probability, and hence $W \notin\MinSum(Z)$.
\end{pf}

\subsection*{Definition of a speed function for plump distributions $Z$}
We are now in a position to partially explain the mysterious function
$h$ defined in (\ref{eqdeff}), which recall was defined by
\[
h(0) = m_0 \quad\mbox{and}\quad h(n+1) = F_{Z}^{-1}
\bigl(1-1/h(n)\bigr).
\]

It will turn out that this function defines a speed function for the
offspring distribution $Z$ in the sense of Definition~\ref{defspeed}.
%
\begin{theorem}\label{thmspeed}
If the offspring distribution $Z$ is plump, then the function $h$ is a
speed of $Z$.
\end{theorem}
Although it is possible to present a proof at this stage, to avoid
redundancy, we postpone it until Section~\ref{secproof}.

It will actually be convenient in our proofs to consider a slight
variation on $h$.
Let $\alpha= (1+\varepsilon)^{-1/2}$, and define $f$ by
%
\begin{equation}\label{eqdefspv}
f(0) = \tilde{m}_0 \quad\mbox{and}\quad f(n+1) = F_{Z}^{-1}
\bigl(1-f(n)^{-\alpha}\bigr),
\end{equation}
where $\tilde{m}_{0}$ is the least integer such that condition (\ref
{eqcondition}) holds
with $m_{0}=\tilde{m}_{0}^{\alpha}$, and the following inequalities hold:
$\tilde{m}_{0}^{1-\alpha}\ge16(1-\alpha)^{-1}+16$ and
$\tilde{m}_{0}^{\alpha^{-1}-1}\ge4^{\lceil(\alpha
^{-1}-1)^{-1}\rceil+1}$.

The functions $h$ and $f$ are essentially equivalent as far as we are
concerned. The following lemma demonstrates their equivalence as speeds.
%
\begin{lemma}\label{lemcomparespeeds}
For any plump distribution $Z$, $h$ is a speed for $Z$ if and only if
$f$ is.
\end{lemma}
\begin{pf}
Since $h$ is increasing, for some constant $c$ we have $h(c)\ge\tilde
{m}_{0}=f(0)$.
Inductively, we then have $f(n) \le h(n+c)$ for all $n$.
Since $Z$ is plump, we have from the definition of $f$ that
\[
f(n+1) \geq f(n)^{\alpha(1+\varepsilon)} = f(n)^{1/\alpha}
\qquad\mbox{for any $n$}.
\]
Thus,
\[
f(n+2) = F_{Z}^{-1}\bigl(1-f(n+1)^{-\alpha}\bigr)
\geq F_{Z}^{-1}\bigl(1-f(n)^{-1}\bigr).
\]
It follows that if $f(n) \geq h(m)$, then $f(n+2) \geq h(m+1)$.
So by induction, we have $f(2n) \geq h(n)$.

Considering the definition of a speed for $Z$, we see that if one is a
speed, so is the other.
\end{pf}

In the following lemma, we state some direct consequences of
condition (\ref{eqcondition}) (i.e., the assumption $Z$ is plump) and
the definition of $f$, that will be helpful later.
%
\begin{lemma}\label{lemexponential}
Let $Z$ be a plump distribution and let $f(n)$ be defined as in~(\ref
{eqdefspv}).
\renewcommand\thelonglist{(\roman{longlist})}
\renewcommand\labellonglist{\thelonglist}
\begin{longlist}
\item\label{itemdoublejump} For all $n$,
%
\begin{equation}
\label{lemdoublejump} f(n+2) \geq F_{Z}^{-1}
\bigl(1-1/f(n)\bigr).
\end{equation}
\item\label{itemfour}\label{itemexpspeed}\label{itemsixteen}
$f(n+1) \ge4^{n+1}f(n)$ for all $n\ge0$.
In particular, $f(n)^{1-\alpha}\ge16n+16$ for all $n\ge1$,
and for any positive $r$, $f(n) = \Omega(r^n)$.
\item\label{itemjumpspeed} For each $k\ge2$ and for all $n$,
%
\begin{equation}
\label{lempower} f\bigl(n+ 2\bigl\lceil\log k/ \log(1+\varepsilon)\bigr
\rceil\bigr)
\geq f(n)^k.\vadjust{\goodbreak}
\end{equation}
\end{longlist}
\end{lemma}
\begin{pf}
Part~\ref{itemdoublejump} follows immediately from the proof of
Lemma~\ref{lemcomparespeeds}.
To prove part~\ref{itemfour}, we begin by noting that the ratio $f(n+1)/f(n)$
is at least $f(n)^{\alpha^{-1}-1}$, as $\alpha(1+\varepsilon)=\alpha^{-1}$.
We therefore prove that $f(n)^{\alpha^{-1}-1}\ge4^{n+1}$ for all $n$.
Let $n_{0}=\lceil(\alpha^{-1}-1)^{-1}\rceil$, and note that
since $\tilde{m}_{0}^{\alpha^{-1}-1}\ge4^{\lceil(\alpha
^{-1}-1)^{-1}\rceil+1}$,
the inequality $f(n)^{\alpha^{-1}-1}\ge4^{n+1}$ holds trivially for
$n\le n_{0}$.
For $n> n_{0}$, the result follows easily by induction as
\begin{eqnarray*}
f(n)^{\alpha^{-1}-1}&\ge&\bigl(4^{n}f(n-1)\bigr)^{\alpha^{-1}-1}=
4^{(\alpha^{-1}-1)n}f(n-1)^{\alpha^{-1}-1}\\
&\ge&4f(n-1)^{\alpha^{-1}-1}.
\end{eqnarray*}
To conclude the proof of part~\ref{itemsixteen}, we have to show
$f(n)^{1-\alpha}\ge16n+16$ for all $n$.
For $n\le(1-\alpha)^{-1}$, we trivially have
\[
f(n)^{1-\alpha}\ge f(0)^{1-\alpha}=\tilde{m}_{0}^{1-\alpha}
\ge16(1-\alpha)^{-1}+16.
\]
For $n\ge(1-\alpha)^{-1}+1$, we have $f(n)^{1-\alpha
}/f(n-1)^{1-\alpha}\ge4$, and the
result easily follows by induction.

To prove part~\ref{itemjumpspeed}, we note that
\[
f(n+2) = F_{Z}^{-1}\bigl(1-1/f(n)\bigr) \geq
f(n)^{1+\varepsilon}.
\]
An inductive argument now easily yields that
\[
f(n+2\ell) \geq f(n)^{(1+\varepsilon)^\ell}
\]
for any $n$ and $\ell$.
It follows that $f(2n) \geq m_0^{(1+\varepsilon)^n}$.
We conclude by setting $\ell= \lceil\log k / \log(1+\varepsilon)
\rceil$.
\end{pf}

\section{Proof of the Equivalence theorem}\label{secproof}

In this section we prove the Equivalence theorem.
We first prove it in the second (technically weaker) form and then
describe how the first form may be deduced.

Let $Z$ be a plump offspring distribution, and let $\varepsilon$ and
$m_{0}$ be
such that condition (\ref{eqcondition}) holds for the triple $Z,
\varepsilon$ and $m_{0}$.
Fix an arbitrary $W\in\MinSum(Z)$. We shall prove that $W\in\Expl
(Z)$ (and the theorem will follow).
We define an algorithm which selects a path in the tree in a very
precise way; then using the properties of~$W$, we
prove that with positive probability this path is an exploding path.
Since, conditioned on survival, the event that there is an exploding
path is a $0$--$1$ event,
this is enough to prove the theorem.

The algorithm\vspace*{2pt} depends on a parameter $\alpha$, defined in the previous section:
$\alpha:=(1+\varepsilon)^{-1/2}$. The reason for this choice of
exponent will be clarified later in the proof.

\begin{figure}[h]
\rule{\textwidth}{0.5pt}
\flushleft\textbf{Algorithm} \textsc{FindPath}:\\
Let $x_0$ be the root of the tree.

\textbf{For} $n=0, 1, 2, \ldots\,$:
\begin{itemize}[--]
\item[--] Consider node $x_n$, which is the lowest node in the candidate
exploding path we are constructing. Let $Y_{n+1}$ denote the number of
children of $x_n$.
\item[--] Order the children of $x_n$ by how many children they in turn have,
from largest to smallest. Let $X_{n+1}:=\lceil(Y_{n+1})^{(1-\alpha
)}/2\rceil$. We define the \textit{options} from $x_n$
to be the first $X_{n+1}$ children of $x_n$ in the ordering.
\item[--] If $X_{n+1}=0$, the algorithm terminates in failure. Otherwise,
of the $X_{n+1}$ choices, pick the option whose edge from $x_n$ has the
smallest weight, and set $x_{n+1}$ to be this child.\vspace*{-10pt}
\end{itemize}
\rule{\textwidth}{0.5pt}
\end{figure}

The analysis of the algorithm, and the proof that it provides with
positive probability an exploding path,
will be based on the following assertion.
%
\begin{claim}\label{claimequiv}
There exists a positive integer $a$ such that, with positive probability,
$Z_{n}\le f(an)$ and $Y_{n}\ge f(n)$ hold simultaneously for all $n\in
\N$,
where $f$ is the function defined in equation (\ref{eqdefspv}).
\end{claim}
Indeed, given this, we may deduce immediately that with positive probability
$Z_{n/a} \leq f(n) \leq Z_{n}$ for all $n\in\N$,
implying that $f(n)$ is a speed of $Z$.
Furthermore, since $X_{n}$, the number of options of $x_{n-1}$, is
defined by
$X_{n}=\lceil Y_{n}^{(1-\alpha)}/2 \rceil$,
there is a positive probability that $X_{n}\ge f(n-\gamma)$ for all
$n\in\N$, where $\gamma=2\lceil\log{(1-\alpha)^{-1}}/\log
(1+\varepsilon)\rceil+1$
[this follows from Lemma~\ref{lemexponential}\ref{itemjumpspeed}].

We now observe that, conditional on the inequality $X_{n}\ge f(n-\gamma
)$ holding for all $n\in\N$,
the path $x_{0},x_{1},x_{2},\ldots$ is an exploding path almost surely.
The distribution of the sum of
weights along the path $x_{0},x_{1},x_{2},\ldots\,$,
dependent on $X_{1},X_{2},X_{3},\ldots\,$, is given by
\[
\sum_{n\ge1}\min\bigl\{ W_{n}^{1},\ldots,W_{n}^{X_{n}} \bigr\},
\]
where the $W_{n}^{j}$ are i.i.d. with distribution $W$.
Thus, conditional on the event that $X_{n}\ge f( n-\gamma)$ for all
$n\in\N$,
this sum is stochastically smaller than $\sum_{n\ge1}\min\{
W_{n}^{1},\ldots,W_{n}^{f(n-\gamma)} \}$.
Moreover, Lemma~\ref{lemminsumseq} implies that $W$ is $\sigma
$-summable for the sequence
$\sigma=(f(n))_{n\in\N}$, and since the contribution of any finite
number of terms is finite,
$W$ is also $\sigma$-summable for the sequence $\sigma=(f(n-\gamma
))_{n\in\N}$. This proves that $x_{0},x_{1},x_{2},\ldots$ is an
exploding path almost surely.

So it remains to prove Claim~\ref{claimequiv}, which we will do for
the choice $a=3+2\lceil\log{2}/\log(1+\varepsilon)\rceil$.

Define the two families of events $\{A_n\}_{n \geq1}$ and $\{B_n\}
_{n\geq1}$ by
\[
A_n:= \bigl\{Y_n < f(n) \bigr\},\qquad B_n:= \bigl
\{Z_n > f(an) \bigr\}.
\]
We are led to prove that there is a positive probability that none of
the events $A_{n}$ or $B_{n}$ occur.\vadjust{\goodbreak}
Let $C=A_{1}^c \cap B_{1}^{c}$. The definition of $f$ implies that $Z$
assigns a positive probability to the range $[f(1),f(a)]$, so that $\Pr
\{C\}>0$.
We will show below that
%
\begin{eqnarray}
\label{ineq1} \Pr\{A_{2}\mid C \} &\leq&1/16 \quad\mbox{and}\quad \Pr\bigl
\{A_{n+1}\mid A_{n}^{c} \bigr\}\leq4^{-n-1}\qquad
\mbox{for } n \geq2;
\\
\label{ineq2} \Pr\{B_{2}\mid C \}&\le&1/16 \quad\mbox{and}\quad \Pr\bigl
\{B_{n+1}\mid B_{n}^{c} \bigr\}\le4^{-n-1}
\qquad\mbox{for } n\ge2.
\end{eqnarray}
Assuming the above inequalities, we infer that
\begin{eqnarray*}
\Pr\biggl\{ C\cap\bigcap_{n\ge1}A_{n+1}^{c}
\biggr\} &=& \Pr\{C\} \prod_{n\ge1}\Pr\bigl\{
A_{n+1}^{c} \mid A_{n}^{c},A_{n-1}^{c},\ldots,A_{2}^{c},C \bigr\}
\\
&=& \Pr\{C\} \Pr\bigl\{ A_{2}^{c} \mid C \bigr\} \prod
_{n\ge
2} \Pr\bigl\{ A_{n+1}^{c}
\mid A_{n}^{c} \bigr\}
\\
&&\mbox{(since the sequence $Y_{1},Y_{2},Y_{3},\ldots$ is Markovian)}
\\
&\geq&\biggl(1 - \sum_{n\geq1} 4^{-n-1}
\biggr)\Pr\{C\}.
\end{eqnarray*}
In the\vspace*{1pt} same way, we obtain $\Pr\{ C\cap\bigcap_{n\ge
1}B_{n+1}^{c} \} \geq(1 - \sum_{n\geq1} 4^{-n-1}
)\Pr\{C\}$.
Since both the events $C\cap\bigcap_{n\ge1}A_{n+1}^{c}$ and $C\cap
\bigcap_{n\ge1}B_{n+1}^{c}$ are
contained\vspace*{1pt} in~$C$, we conclude that with positive probability
none of the events $A_n$ and $B_n$ occur, finishing the proof of the claim.

All that remains is to prove inequalities (\ref{ineq1}) and (\ref{ineq2}).
We first prove the bound on $\Pr\{A_{n+1} | A_{n}^{c} \}$
(it will be seen that the bound on $\Pr\{A_{2} | C \}$
follows by the same proof).
Call a child of $x_{n}$ \textit{good} if it has at least $f(n+1)$
children, and write $G_{n}$ for the number of good children of $x_{n}$.
We note that,
given~$Y_n$, the distribution of $G_{n}$ is $\Bin(Y_n,p)$, where $p$,
the probability that a given child is good, is at least
$1 - F_{Z}(f(n+1)) = f(n)^{-\alpha}$.
By the way the algorithm chooses the vertex $x_{n+1}$, we also note
that $A_{n+1}$ can occur only
if $G_n < Y_n^{1-\alpha}/2$.
Thus, conditional on $Y_n \geq f(n)$, if $A_{n+1}$ occurs, then
\[
G_{n} < Y_n^{1-\alpha}/2 \le Y_nf(n)^{-\alpha}/2
\le\E\{ G_{n}\}/2.
\]
Hence,
\begin{eqnarray*}
\Pr\bigl\{A_{n+1} \mid A_n^c \bigr\}&\le&\Pr
\biggl\{ G_{n}\le\frac{Y_n^{1-\alpha}}{2} \Big| Y_{n} \geq f(n)
\biggr\}
\\
& 
\le&\exp\biggl(\frac{-f(n)^{1-\alpha}}{8} \biggr)
\\
&\le& \frac{1}{4^{n+1}} \qquad\mbox{[by Lemma
\ref{lemexponential}\ref{itemsixteen}]}.
\end{eqnarray*}

We now prove $\Pr\{B_{n+1}\mid B_n^c \}\le4^{-(n+1)}$ (the proof
bounding $\Pr\{B_2 \mid C\}$ being identical).
Note that by Lemma~\ref{lemexponential}\ref{itemjumpspeed},
\[
f(an+a) \geq f(an)f(an+3).
\]
Thus, in order for the event $Z_{n+1} \geq f(an+a)$ to occur,
conditional on $Z_n \leq f(an)$, there must be some node in generation
$n$ having at least \mbox{$f(an+3)$} children.
Taking $Z(i)$ to be an independent copy of $Z$ for each $i$,
the probability of this is bounded as follows:
\begin{eqnarray*}
&&
\Pr\bigl\{\max\bigl\{Z(1),\ldots,Z\bigl(f(an)\bigr)\bigr\}> f(an+3)
\bigr\}\\
&&\qquad\le
f(an) \Pr\bigl\{Z> f(an+3)\bigr\}
\\
&&\qquad\le f(an) \bigl(1-F_{Z}\bigl(f(an+3) \bigr)\bigr)
\\
&&\qquad\le f(an) f(an+1)^{-1}
\qquad\mbox{[by Lemma~\ref{lemexponential}\ref{itemdoublejump}]}
\\
&&\qquad\le\frac{1}{4^{n+1}} \qquad\mbox{[by Lemma~\ref{lemexponential}\ref{itemfour}]}.
\end{eqnarray*}

The proof of the Equivalence theorem (in its second form) is complete.
Note that in the process, we have also proved that $f$
is a speed of $Z$; thus, by Lem\-ma~\ref{lemcomparespeeds},
Theorem~\ref{thmspeed} also follows.

\subsection*{First form of the Equivalence theorem}
One might hope that the first form of the Equivalence theorem could be
deduced from the
second by some very simple reasoning, perhaps considering for each
weight distribution $W$ the set of trees $T$ for which $\Expl(T)\neq
\MinSum(T)$. However, the fact that there are uncountably many
possible weight distributions seems to be problematic for such a direct
approach.

Taking $T$ to be a random Galton--Watson tree with offspring
distribution $Z$ conditioned to survive,
we will prove that the following chain of containments holds almost surely:
\[
\MinSum(T)\subseteq\MinSum(Z)\subseteq\Expl(T).
\]
From this the Equivalence theorem in its first form immediately follows.

That the first inclusion holds almost surely
follows from the fact that the rate of growth of generation sizes of $T$
may almost surely be bounded in terms of the speed $f$ of $Z$.
Specifically, taking $a=3+2\lceil\log{2}/\log(1+\varepsilon)\rceil$
as in Claim~\ref{claimequiv}, we will show that almost surely there
exists a constant $c$ such that $Z_{n}\le f(an+c)$ for all $n$.
For $z \in\mathbb{N}$, let $r(z)$ denote the greatest $r$ for which
$z\ge f(r)$.
If no bound of the form $Z_{n}\le f(an+c)$ holds,
then there must be infinitely many $n$
for which $r(Z_{n+1})> r(Z_{n})+a$.
However, our proof of (\ref{ineq2}) demonstrates that the probability
that $Z_{n+1}\ge f(r+a)$ given that $Z_n\le f(r)$ is at most $4^{-r}$.
Since $f$ is a speed of $Z$, the sequence of probabilities
$4^{-r(Z_n)}$ is summable almost surely, and so this event has
probability zero.

That the second inclusion holds almost surely follows from the fact
that we may apply
the above algorithmic approach to finding an exploding path to any
rooted subtree of $T$
which survives. For a node $v$, let $T_{v}$ denote the subtree of its
descendants.
Denote by $s(n)$ the number of nodes of
generation $n$ for which $T_{v}$ is infinite.\vadjust{\goodbreak}
As $T$ is conditioned on survival, the function $s(n)$
is unbounded almost surely (\cite{AN72}, Chapters 10--12).
Let now $W\in\MinSum(Z)$. The above algorithm, applied independently
to each node of generation $n$ for which
$T_{v}$ is infinite, has positive probability $p>0$ of producing an
exploding path in each.
Thus, the probability of no exploding path is at most
$(1-p)^{s}$ for all $s$, and so is $0$.

\subsection*{The set of weights of infinite rooted paths} The
following theorem characterizes the set of all possible values the
weights of infinite rooted paths can take conditioned on the survival
of the Galton--Watson tree. Note that the theorem is valid in general
and does not require the plumpness condition.
%
\begin{theorem}
Let $Z$ be an offspring distribution and $W$ a nonnegative weight
distribution which is not a.s. zero.
Then almost surely conditioned on survival, the set of weights of
infinite rooted paths is $[A, \infty]$, where $A$ is the infimum
weight of infinite rooted paths.
\end{theorem}
\begin{pf}
By applying the transformation discussed in the \hyperref[intro]{Introduction} if
necessary, we may assume that $W$ has no atom at zero. Note that
clearly the transformation does not change the weights of infinite
rooted paths.

The theorem is clearly true if $W \notin\Expl(Z)$ since in this case,
conditioned on survival, all infinite rooted paths have infinite
weight. So in the following we assume $W \in\Expl(Z)$.

By a straightforward compactness argument, it suffices to show that for
any $\varepsilon' > 0$, there exists (almost surely) an infinite path
with weight in \mbox{$[a, a + \varepsilon']$}, for all $a \geq A$. 

Let $\varepsilon\leq\varepsilon'/4$ be such that $\Pr\{W \in(\varepsilon,
2\varepsilon)\} > 0$; such an $\varepsilon$ must exist since $W \in\Expl
(Z)$ and $W$ has no atom at zero.
Define the \textit{path-weight} $\pathweight(v)$ of a node $v$ to be
the sum of the edge weights on the path from $v$ to the root.
Now let
\[
S_i = \bigl\{ v \in T \mid\pathweight(v) \in\bigl[i\varepsilon, (i+1)
\varepsilon\bigr) \bigr\}.
\]
The choice of $\varepsilon$ is such that if $v \in S_i$, then for any
given child $w$ of $v$, $w \in S_{i+1}\cup S_{i+2}$ with a constant
positive probability. 

Since explosion occurs, there is some least integer $\ell$ such that
$S_{\ell}$ is infinite; we then have $A \geq\ell\varepsilon$. We may
explore $S_0, S_1, \ldots$ in turn, each time uncovering all of~$S_i$,
as well as all children of nodes in $S_i$. In the process of exploring
$S_{\ell}$, each node we explore whose parent is in $S_{\ell}$ will
have a constant positive probability of being in $S_{\ell+1}\cup
S_{\ell+ 2}$, thus, a.s. at least one of $S_{\ell+1}$ and $S_{\ell
+2}$ is infinite too. Moreover, since explosion occurs, each such node
will have a positive probability of being the root of an infinite path
of length at most~$\varepsilon$. Thus, $S_{\ell+1} \cup S_{\ell+2} \cup
S_{\ell+3}$ must contain an infinite path a.s. Continuing inductively,
we find that a.s. for any integer $j\geq\ell$, one of the sets $S_j$
or $S_{j+1}$ should be infinite, and there is an infinite path of total
weight in $[j\varepsilon, (j+4)\varepsilon)$.\vadjust{\goodbreak}

Now choosing $j$ such that $a \in[j\varepsilon, (j+1)\varepsilon)$, we
infer the existence of an infinite path with length in the interval
$[a, a+4\varepsilon] \subseteq[a, a+\varepsilon']$.
\end{pf}

\section{Equivalent conditions for min-summability}\label{secminsummability}

In the previous section, we proved an Equivalence theorem between
explosion and min-summability for
branching processes with plump offspring distributions. Though the
existence of such a result is certainly nice in its own right,
one may wonder if the property of min-summability is in any sense
substantially simpler than that of explosion.
The aim of this section is to answer this question in the affirmative
by proving Theorem~\ref{thmcomputableequivalence}, which
provides a necessary and sufficient condition
for min-summability that involves a calculation based only on the
distributions. We then provide some examples at the end of this section.

Let $W$ be a random variable taking values in $[0, \infty)$ and let
$\sigma=(\sigma_i)_{i \ge0}$ be a sequence of positive integers.
Then we have the following:
%
\begin{proposition}\label{propsigma-equivalence} The nonnegative
random variable $W$ is $\sigma$-summable
if and only if the following two conditions are satisfied:
\begin{eqnarray*}
\mbox{\hphantom{\textup{i}}\textup{(i)}}\hspace*{26pt}\quad \sum_n \bigl( \Pr\{ W > 1 \}
\bigr)^{\sigma_n} &<& \infty\quad\mbox{and}
\\
\mbox{\textup{(ii)}}\quad \sum_n \int_0^1
\bigl( \Pr\{ W > t \} \bigr)^{\sigma_n} \,dt &<& \infty.
\end{eqnarray*}
\end{proposition}
\begin{pf}
As in Section~\ref{secprel}, let $W_n^j$ be an independent copy of
$W$ for each $n,j \in\N$ and
let
\[
\Lambda_n:= \min_{1\leq j\leq\sigma_n} W_n^j.
\]
Clearly, $\Lambda_n$ is a sequence of nonnegative and independent
random variables.
By Kolmogorov's three-series theorem (see, e.g., Kallenberg \cite
{Kal97} or Petrov~\cite{Pet75}),
we have $\sum_n \Lambda_n < \infty$ almost surely if and only if
\begin{eqnarray*}
\sum_n \Pr\{ \Lambda_n > 1 \} &<&
\infty,
\\
\sum_n \E\{ \Lambda_n
\mathbf1_{[\Lambda_n \le1]} \} &<& \infty
\end{eqnarray*}
and
\[
\sum_n \Var\{ \Lambda_n
\mathbf1_{[\Lambda_n \le1]} \} < \infty.
\]
Since $W$ is nonnegative, random variables $\Lambda_n \mathbf
1_{[\Lambda_n \le1]}$ take value in $[0,1]$, and so the third
condition follows from the second one. Now,
$\Pr\{ \Lambda_n > 1 \} = ( \Pr\{ W > 1 \} )^{\sigma_n}$,
and
$\E\{ \Lambda_n \mathbf1_{[\Lambda_n \le1]} \} =
(\int_0^1 ( \Pr\{ W > t \} )^{\sigma_n} \,dt ) - \Pr\{
\Lambda_n > 1\}$,
thus proving the theorem.\vadjust{\goodbreak}
\end{pf}

In the case of a random integer sequence given by the generation sizes,
it is also possible to give a result analogous to Proposition \ref
{propsigma-equivalence} (whose proof is omitted).
%
\begin{proposition}
Let $\{Z_n\}$ be a Galton--Watson process with an offspring distribution $Z$,
satisfying $Z\geq1$ almost surely. Let $\Lambda_n$ be the minimum
weight of the $n$th generation. We have
\[
\Pr\biggl\{ \sum_n \Lambda_n <
\infty\biggr\} = 1
\]
if and only if the following two conditions are satisfied:
\begin{eqnarray*}
\mbox{\textup{\hphantom{i}(i)}}\hspace*{26.5pt}\quad \Pr\biggl\{ \sum_n \bigl( \Pr\{ W >
1 \} \bigr)^{Z_n} < \infty\biggr\}&=&1\quad\mbox{and}
\\
\mbox{\textup{(ii)}}\quad \Pr\biggl\{ \sum_n \int
_0^1 \bigl( \Pr\{ W > t \} \bigr)^{Z_n}
\,dt < \infty\biggr\}&=&1.
\end{eqnarray*}
Otherwise, $\Pr\{ \sum_n \Lambda_n < \infty\} = 0$.
\end{proposition}

The two above propositions are likely the most general form of
necessary and sufficient conditions on min-summability one may hope for.
However, under some extra conditions on the sequence $\sigma$, it is
possible to unify the two conditions of
Proposition~\ref{propsigma-equivalence} into one single and simpler condition.
%
\begin{corollary}\label{corsigmasummability}
Let $\sigma$ be a sequence of integers such that there exists \mbox{$c > 1$}
with the property that for all large enough values of $n$, $\sigma
_{n+1}\geq c\cdot\sigma_n$ (think of the speed function $f$; see
Lemma~\ref{lemexponential}).
Then\vspace*{1pt} $W$ is $\sigma$-summable if and only if $\sum_n F_W^{-1}(\frac
{1}{\sigma_n}) < \infty$.
\end{corollary}
\begin{pf}
Note that, under the assumption of the corollary on the growth of
$\sigma_n$,
condition (i) of Proposition~\ref{propsigma-equivalence} always
holds, provided that $\Pr\{W > 1\} <1$.

Let\vspace*{1pt} $\sigma$ be a sequence satisfying the condition $\sigma_{n+1}\geq
c\cdot\sigma_n$ for all $n$. Let $a_{0}=0$ and
$a_{n}=F_{W}^{-1}(\frac{1}{\sigma_{n}})$ for $n\ge1$, and
suppose\vspace*{1pt}
that $\sum_{n\ge0}a_{n}<\infty$. In this case, trivially $\Pr\{W >
1\} <1$.
We show that condition (ii) of Proposition \ref
{propsigma-equivalence} holds.
We have
\begin{eqnarray*}
\int_0^1 \bigl( \Pr\{ W > t \}
\bigr)^{\sigma_n} \,dt &=& \int_{0}^{a_{n-1}} \bigl(\Pr
\{ W > t \} \bigr)^{\sigma_n} \,dt + \int_{a_{n-1}}^{1}
\bigl( \Pr\{ W > t \} \,dt \bigr)^{\sigma_n}
\\
&\leq& a_{n-1} + \sum_{m=1}^n
a_{m-1} \bigl( \bigl(\Pr\{ W > a_{m} \}
\bigr)^{\sigma_n} - \bigl(\Pr\{W > a_{m-1}\}\bigr)^{\sigma_n}
\bigr)
\\
&\leq& a_{n-1} + \sum_{m=1}^n
a_{m-1}(1-1/\sigma_m)^{\sigma_n}.
\end{eqnarray*}
Thus,
\begin{eqnarray*}
\sum_n \int_0^1
\bigl( \Pr\{ W > t \} \bigr)^{\sigma_n} \,dt &\leq& \sum
_n a_n + \sum_m
a_{m-1}\sum_{n\geq m} (1-1/
\sigma_m)^{\sigma
_n}
\\
&\leq& \sum_n a_n + \sum
_m a_{m-1}\sum_{n\geq m}
(1-1/\sigma_m)^{c^{n-m}\sigma_m}
\\
&\leq& \sum_n a_n + \sum
_m a_{m-1} \sum_{j=0}^\infty
e^{-c^j}
\\
&=& O(1)\sum_n a_n < \infty.
\end{eqnarray*}
This shows that $W$ is $\sigma$-summable.

To prove the other direction, suppose that $W$ is $\sigma$-summable,
so that by Proposition~\ref{propsigma-equivalence},
\[
\sum_n \int_0^1
\bigl( \Pr\{ W > t \} \bigr)^{\sigma_n} \,dt < \infty.
\]

Since $W$ is $\sigma$-summable, we have $F_W(1) >0$ and so there
exists an integer $N$ such that for $n\geq N$, $a_n \leq1$. Thus,
\begin{eqnarray*}
\sum_n \int_0^1
\bigl( \Pr\{ W > t \} \bigr)^{\sigma_n} \,dt &\geq& \sum
_{n\geq N} \int_0^{a_{n}} \bigl( \Pr
\{ W > t \} \bigr)^{\sigma_n} \,dt
\\
&\geq& \sum_{n\geq N} \int_0^{a_{n}}
\bigl( 1- \Pr\{ W \leq a_{n} \} \bigr)^{\sigma_n} \,dt
\\
&=& \sum_{n\geq N} \int_0^{a_{n}}
\biggl( 1- \frac1{\sigma_{n}} \biggr)^{\sigma_n} \,dt
\\
&=& \Omega(1)\sum_{n\geq N} a_n.
\end{eqnarray*}
It follows that $\sum_n a_n <\infty$ and the corollary follows.
\end{pf}

Combining the above corollary with Theorem~\ref{thmspeed} and
Proposition~\ref{propsigma-addable},
we infer a proof of Theorem~\ref{thmcomputableequivalence}.

\subsection*{Examples and special cases} Here we give a family of
examples of applications of Proposition~\ref{propsigma-equivalence}.
The notation is that of Proposition~\ref{propsigma-equivalence}. (In
particular, $\Lambda_n$~is the minimum of
$\sigma_n$ copies of the weight distribution $W$.)
\begin{longlist}
\item
If $W \ge a > 0$, then condition (ii) of Proposition \ref
{propsigma-equivalence} does not hold, and so
$\sum_n \Lambda_n = \infty$. (This also trivially follows from
$\Lambda_n \ge a$.)\vadjust{\goodbreak}
This example shows that the only interesting cases occur when $0$ is an
accumulation point of the
distribution.
\item
If $W=0$ with probability $p > 0$, then both the conditions of
Proposition~\ref{propsigma-equivalence} hold
if $\sum_n (1-p)^{\sigma_n} < \infty$. On the other hand, $\sum_n
\Lambda_n < \infty$
implies that $\sum_n (1-p-\varepsilon)^{\sigma_n} < \infty$ for every
$\varepsilon\in(0,p)$.
This case is not of prime interest either. The case $p=0$ with $0$
being an accumulation point of $W$ is the most interesting.
\item
If $W$ is uniform on $[0,1]$, then the conditions of Proposition \ref
{propsigma-equivalence} are equivalent to
\[
\sum_n \frac{1}{\sigma_n + 1} < \infty.
\]
\item
If $W$ is exponential, then $\Lambda_n \inlaw E/\sigma_n$, where
$E$ is exponential. The sequence $\Lambda_n$ has almost surely a
finite sum
if and only if
\[
\sum_n \frac{1}{\sigma_n} < \infty.
\]
\item
For the sequence $\sigma_n = n$, assuming that there is no atom at the
origin and that $0$ is an
accumulation point for $W$, it is easy to verify that $\sum_n \Lambda_n
< \infty$
almost surely if and only if
\[
\int_0^1 \frac{1}{\Pr\{ W > t \}} \,dt < \infty.
\]
\item
For the sequence $\sigma_n \sim c^ n$, with $c > 1$ a positive
constant, and assuming no atom at the origin,
but with $0$ an accumulation point for $W$, it is easy to verify that
$\sum_n \Lambda_n < \infty$
almost surely if and only if
\[
\int_0^1 \ln\biggl( \frac{1}{\Pr\{ W > t \} }
\biggr) \,dt < \infty.
\]
\end{longlist}

\section{Sharpness of the condition in the Equivalence theorem}
\label{seccounterexample}

The main result of this article, the Equivalence theorem, gives a
sufficient condition on a distribution $Z$ for the equality $\Expl(Z)=
\MinSum(Z)$ to occur.
This condition, that for some $\varepsilon>0$ the inequality $\Pr\{Z\ge
m^{1+\varepsilon}\}\ge1/m$ holds for all sufficiently large $m\in
\mathbb{N}$, demands that $Z$ has a heavy tail and, furthermore, that
the tail is consistently heavy. This condition ensures that the
generation sizes (equivalently, the speed) of the corresponding
branching process are at least double exponential. Furthermore, it
ensures that the rate of growth is always at least the rate associated
with double exponential functions [i.e., $f(n+1)\ge f(n)^{1+\varepsilon}$].
It is therefore natural to ask:
\begin{longlist}
\item Could a weaker version of our condition still imply $\Expl
(Z)\!=\!\MinSum(Z)$?
\item Could a lower bound on the speed of $Z$ alone (e.g., $Z$ has a
speed $f$ which is at least double exponential) be sufficient to
guarantee $\Expl(Z)=\MinSum(Z)$?\vadjust{\goodbreak}
\end{longlist}

Theorem~\ref{propinert} answers (i) in the negative (almost completely)
by showing that no substantially weaker version of our condition
implies $\Expl(Z)=\MinSum(Z)$.
Theorem~\ref{prophaste} answers (ii), completely, in the negative.
In a sense, these results show the Equivalence theorem to be best possible.
%
\begin{theorem} \label{propinert}
Let $g\dvtx \N\to\N$ be an increasing function satisfying $g(m)=m^{1+o(1)}$.
Then there is a distribution $Z$, satisfying $\Pr\{Z\ge g(m)\}\ge1/m$
for all $m\in N$, but for which $\Expl(Z)\neq\MinSum(Z)$.
\end{theorem}
%
\begin{theorem}\label{prophaste} Let $s\dvtx \N\to\N$ be any function.
Then there is a function $f\dvtx \N\to\N$, satisfying $f(n)\ge s(n)$ for
all $n\in\mathbb{N}$, and a distribution $Z$ for which $f$ is a
speed, such that $\Expl(Z)\neq\MinSum(Z)$.
\end{theorem}

There does not seem to be an obvious intuitive way to judge, for a
given distribution $Z$, whether the equality $\Expl(Z)=\MinSum(Z)$
should hold or not.
So before giving our proof of Theorem~\ref{propinert}, we establish a
sufficient condition for the equality to fail; see Proposition \ref
{propsuffcond} below.

We recall that a function $f\dvtx \N\to\N$ is a speed of a distribution
$Z$ if there exist $a,b\in\N$ such that with positive probability the
bounds $Z_{n/a}\le f(n)\le Z_{bn}$ hold for all $n$.
We shall say that $f$ is a \textit{dominating} speed if we may take $a=1$.
We shall say that $f$ is \textit{swift} if, for some $c>1$, the
inequality $f(n+1)>cf(n)$ holds for all $n\ge0$. It will be useful
(for technical reasons) to restrict our attention to swift dominating speeds.
The following direct consequence of Corollary \ref
{corsigmasummability} and Proposition~\ref{propsigma-addable} will
be useful in our proof of Proposition~\ref{propsuffcond}.
%
\begin{lemma}\label{leminW0}
Let $Z$ be a distribution with mean greater than $1$, $f$ a swift speed
of $Z$ and $W$ a weight distribution for which the sum $\sum
_{n=1}^{\infty}F^{-1}_{W}(f(n)^{-1})$ is bounded.
Then $W\in\MinSum(Z)$.
\end{lemma}
%
\begin{proposition}\label{propsuffcond}
Let $Z$ be any distribution with a swift dominating speed $f$ satisfying
%
\begin{equation}
\label{eqgoodspeed} \liminf_{n\to\infty} 2^n f(n) f\bigl(\bigl
\lceil n/\omega(n)\bigr\rceil\bigr)^{-n/2} = 0
\end{equation}
for some function $\omega(n)\to\infty$ as $n\to\infty$. Then
$\Expl(Z) \neq\MinSum(Z)$.
\end{proposition}
\begin{pf}
We must prove the existence of a weight distribution $W$ such that
$W\in\MinSum(Z)$ but $W\notin\Expl(Z)$. Before defining $W$, we
first define some sequences on which its definition will be based. From
our assumption on $f$, there exists an increasing sequence $n_i$ such that
%
\begin{equation}
\label{eqsubseq} \lim_{i\to\infty} 2^{n_i} f(n_i) f
\bigl(\bigl\lceil n_i /\omega(n_i)\bigr\rceil
\bigr)^{-n_i/2} = 0.
\end{equation}
Let us define the sequence $\omega_i$ by $\omega_i=\omega(n_i)$ and
the sequence $\beta_i$ by $\beta_i = \sqrt{\omega_i}$. We note that
$\beta_i\to\infty$ as $i\to\infty$, and so we may choose a
subsequence\vadjust{\goodbreak} $\beta_{i_j}$ with the property that $\beta_{i_j}\ge
2^{j}$ for each $j\ge1$. Finally, set $m_i:=\lceil n_i/\omega_i\rceil$.
We now define the weight distribution $W$ to satisfy
\[
\Pr\biggl\{W< \frac{1}{\beta_{i_{j}}m_{i_j}} \biggr\} = \frac
{1}{f(m_{i_j})} \qquad\mbox{for all }
j \geq1
\]
by placing probability mass $f(m_{i_j})^{-1}\!-\!\sum
_{j'>j}f(m_{i_{j'}})^{-1}$ at position $1/\beta_{i_{j+1}}m_{i_{j+1}}$
for each $j\ge1$, and probability mass $1-\sum_{j'\ge1}
f(m_{i_{j'}})^{-1}$ at $1$.

We first observe that $W\in\MinSum(Z)$. Indeed, this follows
immediately from Lemma~\ref{leminW0} and the observation that
\[
\sum_{n\ge1}F_{W}^{-1}
\bigl(f(n)^{-1}\bigr) \le\sum_{j\ge1}
m_{i_j}\cdot\frac{1}{\beta_{i_{j}}m_{i_j}} \le\sum_{j\ge1}
\frac{1}{\beta_{i_j}} \le\sum_{j\ge1} \frac{1}{2^j}
= 1.
\]

We now observe that $W\notin\Expl(Z)$. We must prove that
$\Pr\{E\}<1$, where $E$ denotes the event of an infinite path of finite
weight.
Let $G$ be the event that $Z_{n}\le f(n)$ for all $n \in\N$;
since $f$ is a dominating speed of $Z$, $G$ has positive probability.
Thus, it
suffices to prove that $\Pr\{E \mid G\}=0$.

Let $A_j$ be the event that there exists a path from the root to
generation $n_{i_{j}}$ of weight less than $\beta_{i_{j}}/2$.
The event $E$ may occur only if $A_j$ occurs for all sufficiently large
$j$, so it suffices to prove that $\Pr\{A_j\mid G\}\to0$ as $j\to
\infty$.

For the event $A_j$ to occur there must exist a path from the root to
generation $n_{i_j}$ at least half of whose edges have weight less than
$\beta_{i_j}/n_{i_j}$. Since under event $G$ there are at most
$f(n_{i_j})$ such paths, and for each path there are less than
$2^{n_{i_j}}$ choices for a subset of half its edges, we have
\[
\Pr\{A_j \mid G\} \leq2^{n_{i_j}} f(n_{i_j}) \bigl(
\Pr\{W < \beta_{i_j}/n_{i_j}\} \bigr)^{n_{i_j}/2}.
\]
Since
\[
\Pr\{W < \beta_{i_j}/n_{i_j}\} = \Pr\bigl\{W < 1/(
\beta_{i_j} m_{i_j})\bigr\} = 1/f(m_{i_j}),
\]
it follows from (\ref{eqsubseq}) that $\Pr\{A_j \mid G\} \rightarrow
0$ as required.
\end{pf}
\begin{pf*}{Proof of Theorem~\ref{propinert}}
Let $g$ be any increasing function satisfying the condition of the
theorem, that is, $g(m) = m^{1+o(1)}$.
We define a distribution $Z$ satisfying $\Pr\{Z\ge g(m)\}\ge1/m$ for
all $m\in\N$,
which has a swift dominating speed $f$ satisfying
$\lim\inf_{n\to\infty}2^{n}f(n)f(\lceil n^{1/2}\rceil)^{-n/2}=0$;
the proof is then complete by Proposition~\ref{propsuffcond}.

There is a sense in which it is difficult to achieve these two
objectives simultaneously.
The first asks that $Z$ has a sufficiently heavy tail, while the second
would seem to get more likely to occur if the tail of $Z$ were less
heavy. Our approach to achieving the objectives simultaneously is to
define $Z$ to have a heavy, but not at all smooth, tail. In the
resulting Galton--Watson branching process the growth of generation
sizes does not at all resemble a smooth fast growing function (such as
a double exponential), but instead consists of a number of periods of
exponential growth, each period much longer than all proceeding
periods, and with a multiplicative factor very much larger
[in fact, the lengths will be $(2n_i)_{i\ge1}$ and the multiplicative
factors $(m_i)_{i\ge1}$; these sequences are defined below].

Define $n_{i}=10^{10^{i}}$ for each $i\ge1$, and $\varepsilon
_{i}=1/10n_{i}=10^{-(10^{i}+1)}$.
As $g(m)=m^{1+o(1)}$, there exists, for each $\varepsilon_{i}$, a natural
number $m_{i}$
such that $g(m)\le m^{1+\varepsilon_{i}}$ for all $m\ge m_{i}^{1/2}$.
Furthermore, we may
choose $(m_i)_{i \in\N}$ to in addition satisfy
%
\begin{equation}
\label{msobey} m_{i}\ge16n_{i}^{2}M_{i-1}^{2}
\qquad\mbox{for all } i\ge1,
\end{equation}
where $M_0 = 1$ and $M_{j}:=\prod_{i=1}^{j}m_{i}^{2n_{i}}$ for $j \geq1$.
Next define sequences $(N_{j})_{j \in\N}$ and $(L_{j})_{j \in\N}$ by
\[
N_{j}:= \sum_{i=1}^{j}n_{i}
\quad\mbox{and}\quad L_{j}:= m_{j}\prod
_{i=1}^{j-1}m_{i}^{2n_{i}}.
\]

As we mentioned above, we shall define the distribution $Z$ so that the
growth of generation sizes of $T_{Z}$ consists of a number of periods of
exponential growth, each period much longer than all proceeding periods,
and with a multiplicative factor very much larger. [The $j$th period of
growth will have length (approximately) $2n_{j}$ and multiplicative factor
$m_{j}$.] In this context $L_{j}$ is approximately the generation size at
the start of this $j$th period of growth (in fact, after the first step of
this period) and $M_{j}$ the generation size when it ends (i.e., at the
point at which we shall switch into the next, faster, period of growth).
One may observe that $L_{j}=m_{j}M_{j-1}$; note, however, that $L_{j}$
is much larger than $M_{j-1}$, since (\ref{msobey}) implies that
$m_{j}$ is already much larger.

Define the distribution $Z$ by
\begin{eqnarray*}
\Pr\{Z\ge L_{1}\} &=& 1;
\\
\Pr\bigl\{Z\ge m^{1+\varepsilon_{i}}\bigr\} &=& \frac{1}{m},\qquad
L_{i}^{1/(1+\varepsilon_{i})}<m\le M_{i},\qquad i\ge1;
\\
\Pr\{Z\ge L_{i+1}\} &=& \frac{1}{M_{i}},\qquad i\ge1.
\end{eqnarray*}
It is easily verified that this distribution satisfies $\Pr\{Z\ge
g(m)\}\ge1/m$ for all \mbox{$m\in\N$}.
Now define the function $f\dvtx \N\to\N$ (which will be a speed for $Z$) by
\[
f(n) = L_{i+1}m_{i+1}^{2(n-N_{i})-1} \qquad\mbox{with $i$ chosen
so that } N_i < n \leq N_{i+1}.
\]
It is also quite easily verified that $f$ satisfies (\ref
{eqgoodspeed}), using $\omega(n) = n^{1/2}$. In particular, we
observe that
$f(n_{i})\le L_{i}m_{i}^{2n_{i}}$ and,
since $\lceil n_{i}^{1/2}\rceil-N_{i-1}\ge n_{i-1}$,
we have that $f(\lceil n_{i}^{1/2}\rceil)^{n_{i}/2}\ge
L_{i}m_{i}^{n_{i-1}n_{i}}$. It is also easily observed that $f$ is swift.
Thus, in light of Proposition~\ref{propsuffcond}, all that is
required to complete the proof is to demonstrate that $f$ is a
dominating speed of $Z$.
Though it is conceptually straightforward, the proof is rather long; we
stress that it is really just a technical detail.

We prove that with positive probability the bounds $Z_{n}\le f(n)\le
Z_{4n}$ hold for all $n\in\N$. Let $E$ be the event that $Z_{n}>f(n)$
for some $n$, and let $F$ be the event that $Z_{4n}<f(n)$ for some $n$.
Let us subdivide these events by the minimum $n$ for which the required
inequality fails. Let $E_{n}$ be the event that $n$ is minimal such
that $Z_{n}>f(n)$, and $F_n$ the event that $n$ is minimal such that
$Z_{4n} < f(n)$.
We will show that $\sum_{n \geq1} E_n \leq1/4$ and $\sum_{n \geq1}
F_n \leq1/4$, which will complete the proof.

We have stated that our example is designed to exhibit a number of
periods of exponential growth. Once the number of nodes of a given
generation is much larger than $M_{i-1}$, it is clear that, from this
point on, the growth should always be at least geometric (i.e.,
exponential) with multiple~$m_{i}$. Indeed, among $m\gg M_{i-1}$ nodes,
one expects about $m/M_{i-1}$ to have $L_{i}=m_{i}M_{i-1}$ children.
Considering these children alone, we see that the size of the next
generation should be at least $m_{i}$ times as large.

Our bound on the probability of the event $F$ is therefore relatively
straightforward, requiring us to formalize the above statement. The
bound on the probability of $E$ is more difficult, as we are required
to control all ways in which the process could grow faster.
\begin{claimnonumber*}
$\Pr\{E\} \leq1/4$.
\end{claimnonumber*}
\begin{pf}
We shall define two sequences $p_{i,j,k}$ and $q_{i}$ of probabilities,
corresponding to the probabilities of certain unlikely events (events
that would cause faster than expected growth). We then prove a bound on
the probability of $E$ based on the $p_{i,j,k}$ and $q_{i}$,
specifically that this probability is at most their sum. It then
suffices to bound by $1/4$ the sum $\sum_{i,j,k}p_{i,j,k}+\sum_{i}q_{i}$.

For each triple $i,j,k\in\mathbb{N}_{0}$ such that $i\ge1$, $1\le
j\le
n_{i}-1$ and $0\le k\le4j$, we define $p_{i,j,k}$ to be the probability
that among $M_{i-1}m_{i}^{2j}$ independent copies of $Z$, at least
$M_{i-1}m_{i}^{k/2}$
exceed $M_{i-1}m_{i}^{2j+1-k/2}$. We define $q_{1}$ to be the probability
that $Z\ge m_{1}^{2}$ and,
for $i\ge2$, we define $q_{i}$ to be the probability that among
$M_{i-1}$ copies of $Z$, at least one of them
exceeds $M_{i-1}m_{i}^{3/2}$.

We prove the bound
\[
\Pr\{E\}=\sum_{n\ge1}\Pr\{E_{n}\}\le\sum
_{i,j,k}p_{i,j,k}+\sum
_{i}q_{i}.
\]
Notice that for the event $E_{N_{i-1}+1}$ to occur,
we must have
\[
Z_{N_{i-1}}\le f(N_{i-1})=M_{i-1} \quad\mbox{and}\quad
Z_{N_{i-1}+1}> f(N_{i-1}+1)=M_{i-1}m_{i}^{2}.
\]
This in turn implies that at least one of the nodes in
generation $N_{i-1}$ has more than $M_{i-1}m_{i}^{3/2}$ children [as
$M_{i-1}\le m_{i}^{1/2}$; see condition (\ref{msobey})].
Thus, we may bound for each $i$ the probability of the event
$E_{N_{i-1}+1}$ by $q_{i}$.\vadjust{\goodbreak}

Next, for $n$ of the form $N_{i-1}+j+1$ for some $i\in\N$ and $1\le
j\le n_{i}-1$, we note that the occurrence of $E_{n}$ implies
that
\[
Z_{n-1}\le M_{i-1}m_{i}^{2j} \quad\mbox{and}\quad
Z_{n}> M_{i-1}m_{i}^{2j+2}.
\]
It follows that for some $0\le k\le4j$, there are at least
$M_{i-1}m_{i}^{k/2}$ nodes of generation $n-1$ with more than
$M_{i-1}m_{i}^{2j+1-k/2}$ children. Indeed, if this were not the case,
then we would
have
\begin{eqnarray*}
Z_{n} &\leq& \sum_{k=0}^{4j}
\bigl(M_{i-1}m_i^{k/2}\bigr) \bigl(
M_{i-1}m_i^{2j+3/2-k/2}\bigr)
\\
&=& (4j+1)M_{i-1}^{2}m_{i}^{2j+3/2}
\\
&\leq& M_{i-1}m_{i}^{2j+2} \qquad\bigl[\mbox{since }
(4j+1)M_{i-1}\le4n_{i}M_{i-1}\le
m_{i}^{1/2}\bigr].
\end{eqnarray*}
It easily follows that $\Pr\{E_{n}\}\le\sum_{0\le k\le4j} p_{i,j,k}$.

We now prove the bound $\sum_{i,j,k}p_{i,j,k}+\sum_{i}q_{i}\le1/4$.
By the bounds (\ref{msobey}), it suffices to prove for each triple
$i,j,k\in\mathbb{N}_{0}$ with $i\ge1$, $1\le j\le n_{i}-1$ and $0\le
k\le4j$, that
%
\begin{equation}
\label{eqwanted} p_{i,j,k}\le\bigl(m_{i}/e^{2}
\bigr)^{-M_{i-1}m_{i}^{k/2}/2}
\end{equation}
and
\[
q_{i}\le\frac{M_{i-1}}{m_{i}}.
\]

The bound on $q_{i}$ is trivial; since $1/(1+\varepsilon_i)\ge2/3$, it
follows that
\[
\Pr\bigl\{Z\ge M_{i-1}m_{i}^{3/2}\bigr\}=
\bigl(M_{i-1}m_{i}^{3/2}\bigr)^{-1/(1+\varepsilon
_{i})}\le
m_{i}^{-1}.
\]
We bound the probability $p_{i,j,k}$ (that among $M_{i-1}m_{i}^{2j}$
independent copies of $Z$ at least $M_{i-1}m_{i}^{k/2}$
exceed $M_{i-1}m_{i}^{2j+1-k/2}$) using a union bound. By the familiar
estimate ${s\choose t}\le(es/t)^{t}$, the number of choices of the set of
$M_{i-1}m_{i}^{k/2}$ copies is
\[
\pmatrix{M_{i-1}m_{i}^{2j}
\cr
M_{i-1}m_{i}^{k/2}}\le\bigl(em_{i}^{2j-k/2}
\bigr)^{M_{i-1}m_{i}^{k/2}}.
\]
For each copy of $Z$ we have
\[
\Pr\bigl\{Z >M_{i-1}m_{i}^{2j+1-k/2}\bigr\} =
\bigl(M_{i-1}m_{i}^{2j+1-k/2}\bigr)^{-1/(1+\varepsilon_{i})}\le
m_{i}^{-(2j+1/2-k/2)},
\]
where for the final inequality we have used that $\varepsilon_i=1/(10n_i)$
and (since $2j+1/2-k/2\le2n_i$)
\[
2j+1-k/2=2j+1/2-k/2+1/2\ge(2j+1/2-k/2) \bigl(1+1/(4n_i)\bigr).
\]
Thus, the probability that a given set of $M_{i-1}m_{i}^{k/2}$ copies of
$Z$ all exceed $M_{i-1}m_{i}^{2j+1-k/2}$ is at most
\[
m_{i}^{-(2j+1/2-k/2)M_{i-1}m_{i}^{k/2}},
\]
and (\ref{eqwanted}) now follows by a union bound.
\end{pf}
%
\begin{claim}
$\sum_{n\ge1}\Pr\{F_{n}\}\le1/4$.
\end{claim}
\begin{pf}
Our approach is similar to that used in the previous proof. For $i\ge
1$ and $2\le j\le4n_{i}$, we define $p_{i,j}$
to be the probability that from a collection of $M_{i-1}m_{i}^{j/2}$
copies of $Z$, fewer than $M_{i-1}m_{i}^{j/2-1/2}$ exceed $m_{i}$. For
each $i\ge1$, we define $q_{i}$ to be the
probability that the maximum of $M_{i}m_{i}^{1/2}$ copies of $Z$ is
less than $L_{i+1}$. We prove for $n$ of the form $n=N_{i}+1$ that
\[
\Pr\{F_{n}\}\le p_{i,4n_{i}}+q_{i}+p_{i+1,2}+p_{i+1,3}
\]
and for $n$ of the form $n=N_{i}+k, k=2,\ldots, n_{i+1}$, that
\[
\Pr\{F_{n}\}\le p_{i+1,4k-4}+p_{i+1,4k-3}+p_{i+1,4k-2}+p_{i+1,4k-1}.
\]
It will then suffice to bound by $1/4$ the sum $\sum_{i,j}p_{i,j}+\sum
_{i}q_{i}$. For $n=N_{i}+k$, $k=2,\ldots, n_{i+1}$, if the event
$F_{n}$ occurs, then $Z_{4n-4}\ge f(n-1)=M_{i}m_{i+1}^{2k-2}$ and
$Z_{4n}<f(n)=M_{i}m_{i+1}^{2k}$. The required bound now follows, as the
probability for a given $0 \leq l \leq3$ that $l$ is minimal such that
$Z_{4n-l}<M_{i}m_{i+1}^{2k-l/2}$ is at most $p_{i+1,4k-l-1}$. The case
$n=N_{i}+1$ is similar, differing only in that we do not consider the
events $Z_{4n-l}<M_{i}m_{i+1}^{2k-l/2}$ for $0 \leq l \leq3$, but
rather the events $Z_{4n-3}<M_{i}m_{i}^{1/2}$, $Z_{4n-2}<L_{i+1}$,
$Z_{4n-1}<L_{i+1}m_{i}^{1/2}$ and $Z_{4n}<L_{i+1}m_{i}$.

Finally, we prove the bound $\sum_{i,j}p_{i,j}+\sum_{i}q_{i}<1/4$. It
is trivial, using the inequality $(1-p)^{n}\le e^{-pn}$, that $q_{i}\le
\exp(-\sqrt{m_{i}})$. To bound $p_{i,j}$, we first note that $\Pr\{
Z> m_{i}\}\ge1/M_{i-1}$, so from a collection of $M_{i-1}m_{i}^{j/2}$
copies of $Z$ the distribution for the number exceeding $m_{i}$ is
$\Bin(M_{i-1}m_{i}^{j/2},1/M_{i})$. Since this binomial has expected
value $m_{i}^{j/2}\ge2M_{i-1}m_{i}^{j/2-1/2}$, an application of
Chernoff's inequality yields
\[
p_{i,j}\le\exp\biggl(\frac{-m_{i}^{j/2}}{8} \biggr).
\]
\upqed\end{pf}
The proof of Theorem~\ref{propinert} is now complete.
\end{pf*}

The proof of Theorem~\ref{prophaste} is essentially identical to the above.
The only change required is that the following extra condition should
be included in~(\ref{msobey}):
\[
m_{i}\ge\max_{n\le n_{i}}s(n),\qquad i\ge1.
\]
This ensures that the inequality $f(n)\ge s(n)$ holds for all $n\in\N$.
Since the proofs that $f$ is a speed of $Z$ and that $\Expl(Z)\neq
\MinSum(Z)$ are unaffected by this change, Theorem~\ref{prophaste}
does indeed follow.

\section{Limit theorem in the case of no explosion}\label{secconvergence}
So far we only considered the appearance of the event of explosion.
In this section we consider the case of weight distributions for a
heavy-tailed branching random walk
for which explosion does not happen, and obtain a precise limit theorem
for the minimum displacement
$M_n$ under some quite strong (smoothness) assumption on the tails of
$Z$. To explain this, let $Z$ be a plump random variable, and denote by
$G_Z(\cdot)$ the moment generating function of $Z$ as before. Note that
%
\begin{eqnarray}
\label{eqtaub-moment}
K_Z(s)&=& 1- G_Z(1-s)
= \sum_{k=0}^\infty\bigl( \Pr\{Z=k\} -
(1-s)^k\Pr\{Z=k\} \bigr)
\nonumber\\
&=& s \sum_{k=1}^{\infty} \Pr\{Z=k\}
\bigl(1+\cdots+(1-s)^{k-1}\bigr)
\\
&=& s \Biggl(1-\Pr\{Z=0\} + \sum_{k=1}^\infty(1-s)^k
\bigl(1-F_Z(k)\bigr) \Biggr).\nonumber
\end{eqnarray}

Consider now the smoothness condition (\ref{eqsmooth}) on $Z$:
\[
1-F_Z(k) = k^{-\eta}\ell(k)
\]
for some function $\ell$ which is continuous-bounded-and-nonzero at
infinity. In particular, note that one can define
$\ell(\infty) \neq0,\infty$. Using equation (\ref{eqtaub-moment}) and
applying a Tauberian theorem
(see, e.g., Feller~\cite{Feller}, Section XIII. 5, Theorem~5), we see
that condition (\ref{eqsmooth}) is equivalent to
the condition
%
\renewcommand{\theequation}{\mbox{$\star$}}
\begin{equation}
\label{eqmoment} K_Z(s) \sim a s^\eta
\ell\biggl(\frac1{s}\biggr)
\end{equation}
near $s=0$ for some $a>0$ [indeed, $a = \Gamma(1-\eta)$].
This, in particular, implies that $Z$ is plump and
%
\renewcommand{\theequation}{\mbox{$\star\star$}}
\begin{equation}\label{star2}
F_Z^{-1}\biggl(1-\frac1m\biggr) =
m^{1+\varepsilon}\tilde\ell(m)
\end{equation}
for a slowly growing function $\tilde\ell$ and $1+\varepsilon=\eta^{-1}$.
We have the following:
%
\begin{theorem}\label{thmaddspeed}
Let $Z$ be an offspring distribution satisfying (\ref{eqmoment}). Let
$W$ be a nonnegative weight distribution and assume
that $W \notin\Expl(Z)$. Conditional on the survival of the
Galton--Watson process,
\[
\lim_{n\rightarrow\infty} \frac{M_n}{\sum_{k=1}^{n} F_W^{-1}
(1/{h(k)} )}=1.
\]
Here $h(k) = \exp( (1+\varepsilon)^k )$, where $\varepsilon$
is as in (\ref{star2}) and $\eta=(1+\varepsilon)^{-1}$ as in (\ref{eqmoment}).
\end{theorem}

The proof will essentially use the algorithm we presented in
Section~\ref{secproof}.
However, we first need to obtain more precise information on the speed
of the Galton--Watson tree under condition (\ref{eqmoment}).
%
\begin{definition}[(Additive speed)]\label{defaddspeed}
An increasing function $h\dvtx \mathbb N \rightarrow\mathbb R^+$ is an
\textit{additive} speed for a Galton--Watson
offspring distribution $Z$ if the probability of the increasing events
$E_r$ defined as
\[
E_r:= \bigl\{ h(n-r) \leq Z_n \leq h(n+r) \mbox{ for
all large enough $n$} \bigr\}
\]
tend to one as $r$ goes to infinity
conditional on survival.
\end{definition}
%
\begin{lemma}\label{lemaddspeed}
Let $Z$ be an offspring distribution satisfying condition (\ref{eqmoment}).
Then the function $h\dvtx \mathbb N\rightarrow\mathbb R^+$ defined by $h(n)
= \exp( (1+\varepsilon)^k )$ is an additive speed for~$Z$.
\end{lemma}
\begin{pf*}{Proof of Theorem~\ref{thmaddspeed}}
Since $h(n)$ is an additive speed for $Z$, we obtain by Lemma \ref
{lemaddspeed} that, conditional on survival,
\[
\lim_{r\rightarrow\infty} \Pr\{E_r\} = 1.
\]

Fix the integer $r$ and suppose the event $E_r$ holds. This means \mbox{$Z_n
\leq h(n+r)$} for
large enough $n$.
This implies that the minimum of level $n$ is at least $F_W^{-1}(\frac
1{h(n+r)})$ for all large enough $n$.
Since by our Equivalence theorem we have a.s.
$\sum F_W^{-1}(1/h(n)) =\infty$, we obtain
\[
\liminf_{n\rightarrow\infty} \frac{M_n}{\sum_{k=1}^{n}
F_W^{-1} (1/{h(k)} )} = \liminf_{n\rightarrow\infty}
\frac{M_n}{\sum_{k=1}^{n}
F_W^{-1} (1/{h(k+r)} )}\geq1
\]
on $E_r$. We infer that on the union of $E_r$, that is, on the event of
nonexctinction, we have
\[
\liminf_{n\rightarrow\infty} \frac{M_n}{\sum_{k=1}^{n}
F_W^{-1} (1/{h(k)} )}\geq1.
\]
We now show that on the union of $E_r$, we have
\[
\limsup_{n\rightarrow\infty} \frac{M_n}{\sum_{k=1}^{n}
F_W^{-1} (1/{h(k)} )}\leq1.
\]
This will finish the proof of the theorem above.

It will be enough to show this on each $E_r$.
In addition, we can also fix an $n_0$ and suppose that for all $n\geq n_0$,
we have $Z_n \geq h(n-r)$ (and then make $n_0$ tend to infinity). Fix a
small $\delta>0$.
One can now apply a variant of the algorithm of Section~\ref{secproof},
by modifying $\alpha$ to $(1+\varepsilon)^{-\delta}$,
started at some large $N > n_0$,
and show that w.h.p., as $N$ goes to infinity, we have
for all $n \geq N$, $X_n\geq h((1-\delta)n)$ [this follows from a
variant of
the inequalities (\ref{ineq1}) and~(\ref{ineq2})].
In addition, given the double exponential growth of $h(n)$, a union
bound argument shows that we can assume with height probability that
for large enough $n$,\vadjust{\goodbreak} the weight of the $n$th edge on the
path constructed in the algorithm is bounded above by
$F_W^{-1}(1/h((1-2\delta)n))$.
Applying now the Equivalence theorem, since both $M_n$ and $\sum
_{k=1}^{n} F_W^{-1} (\frac1{h((1-2\delta)k)})$ tend to infinity,
we obtain that
\[
\limsup_{n\rightarrow\infty} \frac{M_n}{\sum_{k=1}^{n}
F_W^{-1} (1/{h((1-2\delta)k)} )}\leq1.
\]
Since this holds for
any small enough $\delta>0$, and since the function $F_W^{-1}(1/m)$ is
a decreasing function of $m$, a simple argument shows that
\[
\limsup_{n\rightarrow\infty} \frac{M_n}{\sum_{k=1}^{n}
F_W^{-1} (1/{h(k)} )} = \lim_{\delta\rightarrow0}
\limsup_{n\rightarrow\infty} \frac
{M_n}{\sum_{k=1}^{n} F_W^{-1} (1/{h( (1-2\delta)k )}
)} \leq1.
\]
The theorem follows.
\end{pf*}
\begin{pf*}{Proof of Lemma~\ref{lemaddspeed}}
Under some extra conditions on $\ell$ as in Seneta~\cite{Sen69}
or~\cite{Sen73},
a combination of the results of Darling~\cite{Dar70} and Cohn \cite
{Coh77} with the above mentioned
results of Seneta~\cite{Sen69,Sen73} ensures the
existence of a limiting random variable $V$
such that
\[
(1+\varepsilon)^{-n} \log(Z_n+1) \rightarrow V \qquad\mbox{almost surely}
\]
for $V$ having a strictly increasing continuous distribution $v$, $V >
0$ a.s.
on the set of nonextinction of the process,
and $v(0+)=q$, where $q$ is the extinction probability of the
Galton--Watson process.
In the general case of a function $\ell$ continuous bounded and
nonzero at infinity,
the above limit theorem still holds, as we now briefly explain by
following closely Bramson's strategy in~\cite{Bra78}.
Define $\alpha= 1+\varepsilon=\eta^{-1}$.
The general idea in proving such a limit theorem is to prove
first the convergence of the sequences $K^{(n)} (\exp(-\alpha^ns) )$
uniformly on compact sets.
Here,
$K^{(n)}(\cdot) =K^{(n)}_Z(\cdot) = K_{Z_n}(\cdot)$ is the
$n$-times composition of $K_Z$ [and $K_Z$ is as in equation (\ref
{eqtaub-moment})].
For this, define
\[
H(s):= -\log K \bigl(\exp(-s) \bigr)
\]
and notice that $H^{(n)}(s) = -\log K^{(n)} (\exp(-s) )$,
so that we are left to prove the convergence of
the sequence $H^{(n)} (\alpha^n s)$ as $n$ goes to infinity, for
$s\geq0$.

By an abuse of the notation [from condition (\ref{eqmoment})],
assume that $K_Z(s) = s^\eta\ell(\frac1s)$ for a function $\ell$
continuous bounded and nonzero at infinity,
and define
\[
L(s) = -\log\ell\bigl(\exp(s)\bigr).
\]
By the assumptions on $\ell$, it follows that $L$ is continuous at
infinity and $L(\infty) \neq\pm\infty$,
and so for each $a>0$,
there is an $N_a$
such that for $s_1$ and $s_2$ larger than $N_a$, we have $|L(s_1) -
L(s_2)| \leq a$.
A simple induction shows that
%
\setcounter{equation}{1}
\renewcommand{\theequation}{\arabic{section}.\arabic{equation}}
\begin{equation}
\label{eqdefH} H^{(m)}\bigl(\alpha^ms\bigr) = s + \sum
_{k=1}^m \frac1{\alpha^{m-k}}
(-1)^k L \bigl( H^{(k-1)} \bigl(\alpha^{m-k+1}s\bigr)
\bigr).
\end{equation}

By the definition of $H$, one can easily verify that $H$ is
1-Lipschitz, that~is,
\[
\mbox{for any two $s_1, s_2\geq0$}\qquad
\bigl|H(s_1)-H(s_2)\bigr| \leq|s_1-s_2|.
\]

We now show that the sequence $\{H^{(n)}(\alpha^n s), n\in\mathbb
N\}$ is Cauchy,
proving the point-wise convergence.
The same argument shows that the
sequence is uniformly Cauchy on compact intervals of $[ 0,\infty)$,
concluding the proof of the uniform convergence.

Fix a large $m \in\mathbb N$ and note that replacing $s$ by $\alpha^{n} s$
in (\ref{eqdefH}), we get
\[
H^{(m)}\bigl(\alpha^{n+m}s\bigr) = \alpha^n s +
\sum_{k=1}^m \frac1{\alpha^{m-k}}
(-1)^k L \bigl( H^{(k-1)} \bigl(\alpha^{m-k+1+n}s\bigr)
\bigr).
\]

We claim that as $n$ goes to infinity each term $H^{(k-1)} (\alpha
^{m-k+1+n}s)$ tends to infinity.
Indeed, more precisely, the rate of convergence to infinity of this
term is as
$\alpha^{n+m-2k+2}s + O(1)$; this can be shown by a simple induction
from (\ref{eqdefH}), using
the bounded continuity of $L$ at infinity.

For two fixed $m$ and $M$, we have
\begin{eqnarray*}
&&
\bigl|H^{(m)}\bigl(\alpha^{n+m}s\bigr) - H^{(M)}\bigl(
\alpha^{n+M}s\bigr)\bigr| \\[-2pt]
&&\qquad= \Biggl|\sum_{k=1}^m
\frac1{\alpha^{m-k}} (-1)^k L \bigl(H^{(k-1)}
\bigl(\alpha^{m-k+1+n}s\bigr) \bigr)
\\[-2pt]
&&\qquad\quad\hspace*{2pt}{}- \sum_{k=1}^M \frac1{
\alpha^{M-k}} (-1)^k L \bigl(H^{(k-1)} \bigl(
\alpha^{M-k+1+n}s\bigr) \bigr) \Biggr|.
\end{eqnarray*}
For $n$ large enough, we can assume that each term $L (H^{(k-1)}
(\alpha^{(m-k+1+n)}s) )$ differs from
$L(\infty)$ by an arbitrary small positive number $a$. It follows then
\begin{eqnarray*}
&&\bigl| H^{(m)}\bigl(\alpha^{n+m}s\bigr) - H^{(M)}\bigl(
\alpha^{n+M}s\bigr) \bigr| \\[-2pt]
&&\qquad\leq a \Biggl[ \sum_{k=1}^m
\frac1{\alpha^{m-k}} +\sum_{k=1}^M
\frac1{\alpha^{M-k}} \Biggr]
\\[-2pt]
&&\qquad\quad{}+ \Biggl|\sum_{k=1}^m \frac1{
\alpha^{m-k}} (-1)^k L(\infty) - \sum
_{k=1}^M \frac1{\alpha^{M-k}}
(-1)^k L (\infty) \Biggr|.
\end{eqnarray*}
Since $\alpha>0$ and $L(\infty) <\infty$, and $a$ can be chosen
arbitrarily small, obviously
the right term of the above inequality can be made
arbitrarily small, provided that $n$ is sufficiently large and the
constants $m$ and $M$ are large enough.
We conclude that for any $a>0$, there exist integer constants $N_a$ and
$M_a$ such that
\begin{eqnarray*}
\bigl|H^{(n+m)}\bigl(\alpha^{n+m}s\bigr) - H^{(n+M)}\bigl(
\alpha^{n+M}s\bigr) \bigr| &\leq& \bigl|H^{(m)}\bigl(
\alpha^{n+m}s\bigr) - H^{(M)}\bigl(\alpha^{n+M}s
\bigr) \bigr|
\\[-2pt]
&\leq& a
\end{eqnarray*}
for any $n$ larger than $N_a$, provided that $m$ and $M$ are larger
than $M_a$.
This shows that the sequence is Cauchy.\vadjust{\goodbreak} In the same way, we can easily prove
that the sequence is uniformly Cauchy on compact subsets of $[
0,\infty)$.
This shows the existence of a continuous limit $w$ for the sequence
$H^{(n)}(\alpha^ns)$.

We now show that $w$ is strictly increasing and $w(\infty) =\infty$.
For this, note that for $s_1 <s_2$,
the above arguments show that for large enough $m$ and~$n$,
one has $H^{(m)}(\alpha^{n+m}s_i) = \alpha^n s_i + O(1)$.
In particular, for $n$ large enough constant and for all $m$,
$H^{(m)}(\alpha^{n+m}s_2) - H^{(m)}(\alpha^{n+m}s_1) > \frac12
\alpha^n(s_2-s_1)$.
Since $H$ is itself strictly increasing, and so $H^{(n)}$ is,
one concludes that the limit $w$ is strictly increasing. A~similar
argument shows that $w(\infty) =\infty$.

Finally, we observe that $w(0^+) = -\log(1-q)$.
This follows from a simple fixed point argument: fix an $s>0$ and note that
\begin{eqnarray*}
w(0+)&=&\lim_{m\rightarrow\infty} w \bigl(\alpha^{-m}s\bigr) =
\lim_{m
\rightarrow\infty} \lim_{n\rightarrow\infty} H^{(m)}H^{(n-m)}\bigl(
\alpha^{n-m}s\bigr)
\\
&=&\lim_{m\rightarrow\infty} H^{(m)}\bigl(w(s)\bigr)
\end{eqnarray*}
by the continuity of $H^{(m)}$ for each fixed $m$.

Since $H^{(m)}(w(s))=-\log K_Z^{(m)}(\exp(-w(s)))$ and $w(s) \geq0$,
it follows easily that for each $s>0$, when $m$ goes to infinity,
$H^{(m)}(w(s))$ tends to the unique finite fixed point of $H$.
This is $-\log(1-q)$, a consequence of the corresponding statement for
$K^{(m)}$ given that the unique fixed point of $K_Z$ in $(0,1)$ is $1-q$.

These then allow us to conclude the proof of the above convergence result
by first proceeding as in Darling~\cite{Dar70} to obtain the
convergence in distribution, and next by applying the result of
Cohn~\cite{Coh77} to obtain the almost sure convergence.

To conclude the proof of the lemma, note that for two constants
$\delta, \Delta>0$, $\delta< \Delta$, the event
\[
E_{\delta,\Delta}:= \bigl\{ \delta(1+\varepsilon)^n \leq\log
(Z_n+1) \leq\Delta(1+\varepsilon)^n \mbox{ for large
enough $n$} \bigr\}
\]
happens with a probability tending to $1-q$ as $\delta\rightarrow0$
and $\Delta\rightarrow\infty$. For two fixed
constants $\delta$ and $\Delta$, we have
for $r$ large enough, $(1+\varepsilon)^{-r} \leq\delta$ and
$(1+\varepsilon)^r \geq\Delta$. This shows that
the event $E_{\delta,\Delta}$ is contained in the event $E_r$ for $r$
sufficiently large, and the lemma follows.
\end{pf*}

\section{Conclusion}\label{secconclusion}

We have proved the equivalence of $\Expl(Z)$ and $\MinSum(Z)$ for
plump offspring distributions $Z$, and shown that the plumpness
condition is essentially best possible, in terms of conditions of
the form $F_Z(1-1/m) \geq g(m)$. However, this is very far from being a
characterization of all offspring distributions for which explosion and
min-summability are equivalent. 
For example, a~simple adaptation of the proof of the Equivalence
theorem shows that $\Expl(Z) = \MinSum(Z)$ for $Z$ defined by
\[
\mathbb{P} \biggl\{Z\ge m\exp\biggl(\exp\biggl(\log\log{m}-\sqrt{\log
\log{m}}+\frac{1}{2}\log\log\log{m} \biggr) \biggr) \biggr\} =
\frac{1}{m}.\vadjust{\goodbreak}
\]
The function
\[
f(n)=e^{e^{\log^2{n}}}
\]
is a speed of $Z$. This illustrates that the equivalence can occur for
distributions with speeds very much slower than doubly exponential.
By contrast, any plump distribution has a speed that grows at least as
fast as a double exponential.

We remark that the above example is extremely close to best
possible. It follows from Proposition~\ref{propsuffcond} that the
equivalence cannot hold for an offspring distribution which has a speed
of the form
\[
f(n)=e^{e^{o(\log^2{n})}}.
\]

We do not know how general the equivalence of $\Expl(Z)$ and $\MinSum
(Z)$ should be when $Z$ has speed slower than doubly exponential.
Obtaining a complete characterization of offspring distributions where
equivalence occurs remains an interesting open question.


\section*{Acknowledgments}

The paper was conceived during two consecutive meetings at the Bellairs
Institute in Barbados in 2008 and 2009. Discussions with Ralph
Neininger regarding the limit law of $M_{n}$ are gratefully
acknowledged. We are grateful to the anonymous referees for their
constructive comments and suggestions which helped to significantly
improve the presentation of this paper.
We thank Vladimir Vatutin for making us aware of the references \cite
{Sev67,Sev70,Gel67,Vat76,Vat87,Vat96,VZ93}.
We thank Louigi Addario-Berry for helping to clarify an issue in the literature.



\printaddresses


\begin{thebibliography}{39}

\bibitem{AdRe09}
\begin{barticle}[mr]
\bauthor{\bsnm{Addario-Berry},~\bfnm{Louigi}\binits{L.}} \AND
  \bauthor{\bsnm{Reed},~\bfnm{Bruce}\binits{B.}}
(\byear{2009}).
\btitle{Minima in branching random walks}.
\bjournal{Ann. Probab.}
\bvolume{37}
\bpages{1044--1079}.
\bid{doi={10.1214/08-AOP428}, issn={0091-1798}, mr={2537549}}
\bptok{imsref}%
\end{barticle}
\endbibitem

\bibitem{Ai11}
\begin{bmisc}[auto:STB|2013/01/29|08:09:18]
\bauthor{\bsnm{A{\"i}dekon},~\bfnm{E.}\binits{E.}}
(\byear{2011}).
\bhowpublished{Convergence in law of the minimum of a branching random
  walk. Preprint. Available at \url{http://arxiv.org/abs/1101.1810}}.
\bptok{imsref}%
\end{bmisc}
\endbibitem

\bibitem{AiSh11}
\begin{bmisc}[auto:STB|2013/01/29|08:09:18]
\bauthor{\bsnm{A{\"i}dekon},~\bfnm{E.}\binits{E.}} \AND
  \bauthor{\bsnm{Shi},~\bfnm{Z.}\binits{Z.}}
(\byear{2011}).
\bhowpublished{The Seneta--Heyde scaling for the branching random walk.
  Preprint. Available at \url{http://arxiv.org/abs/1102.0217}}.
\bptok{imsref}%
\end{bmisc}
\endbibitem

\bibitem{AN72}
\begin{bbook}[mr]
\bauthor{\bsnm{Athreya},~\bfnm{Krishna~B.}\binits{K.~B.}} \AND
  \bauthor{\bsnm{Ney},~\bfnm{Peter~E.}\binits{P.~E.}}
(\byear{1972}).
\btitle{Branching Processes}.
\bseries{Die Grundlehren der Mathematischen Wissenschaften, Band}
\bvolume{196}.
\bpublisher{Springer}, \blocation{New York}.
\bid{mr={0373040}}
\bptok{imsref}%
\end{bbook}
\endbibitem

\bibitem{Bac00}
\begin{barticle}[mr]
\bauthor{\bsnm{Bachmann},~\bfnm{Markus}\binits{M.}}
(\byear{2000}).
\btitle{Limit theorems for the minimal position in a branching random walk with
  independent logconcave displacements}.
\bjournal{Adv. in Appl. Probab.}
\bvolume{32}
\bpages{159--176}.
\bid{doi={10.1239/aap/1013540028}, issn={0001-8678}, mr={1765165}}
\bptok{imsref}%
\end{barticle}
\endbibitem


\bibitem{Big90}
\begin{barticle}[mr]
\bauthor{\bsnm{Biggins},~\bfnm{J.~D.}\binits{J.~D.}}
(\byear{1990}).
\btitle{The central limit theorem for the supercritical branching random walk,
  and related results}.
\bjournal{Stochastic Process. Appl.}
\bvolume{34}
\bpages{255--274}.
\bid{doi={10.1016/0304-4149(90)90018-N}, issn={0304-4149}, mr={1047646}}
\bptok{imsref}%
\end{barticle}
\endbibitem


\bibitem{BiKy97}
\begin{barticle}[mr]
\bauthor{\bsnm{Biggins},~\bfnm{J.~D.}\binits{J.~D.}} \AND
  \bauthor{\bsnm{Kyprianou},~\bfnm{A.~E.}\binits{A.~E.}}
(\byear{1997}).
\btitle{Seneta--{H}eyde norming in the branching random walk}.
\bjournal{Ann. Probab.}
\bvolume{25}
\bpages{337--360}.
\bid{doi={10.1214/aop/1024404291}, issn={0091-1798}, mr={1428512}}
\bptok{imsref}%
\end{barticle}
\endbibitem

\bibitem{BrZe07}
\begin{barticle}[mr]
\bauthor{\bsnm{Bramson},~\bfnm{Maury}\binits{M.}} \AND
  \bauthor{\bsnm{Zeitouni},~\bfnm{Ofer}\binits{O.}}
(\byear{2007}).
\btitle{Tightness for the minimal displacement of branching random walk}.
\bjournal{J. Stat. Mech. Theory Exp.}
\bvolume{7}
\bpages{P07010, 12}.
\bid{issn={1742-5468}, mr={2335694}}
\bptok{imsref}%
\end{barticle}
\endbibitem

\bibitem{BrZe09}
\begin{barticle}[mr]
\bauthor{\bsnm{Bramson},~\bfnm{Maury}\binits{M.}} \AND
  \bauthor{\bsnm{Zeitouni},~\bfnm{Ofer}\binits{O.}}
(\byear{2009}).
\btitle{Tightness for a family of recursion equations}.
\bjournal{Ann. Probab.}
\bvolume{37}
\bpages{615--653}.
\bid{doi={10.1214/08-AOP414}, issn={0091-1798}, mr={2510018}}
\bptok{imsref}%
\end{barticle}
\endbibitem

\bibitem{Bra78}
\begin{barticle}[mr]
\bauthor{\bsnm{Bramson},~\bfnm{Maury~D.}\binits{M.~D.}}
(\byear{1978}).
\btitle{Minimal displacement of branching random walk}.
\bjournal{Z.~Wahrsch. Verw. Gebiete}
\bvolume{45}
\bpages{89--108}.
\bid{doi={10.1007/BF00715186}, issn={0044-3719}, mr={0510529}}
\bptok{imsref}%
\end{barticle}
\endbibitem

\bibitem{Coh77}
\begin{barticle}[mr]
\bauthor{\bsnm{Cohn},~\bfnm{Harry}\binits{H.}}
(\byear{1977}).
\btitle{Almost sure convergence of branching processes}.
\bjournal{Z. Wahrsch. Verw. Gebiete}
\bvolume{38}
\bpages{73--81}.
\bid{mr={0433620}}
\bptok{imsref}%
\end{barticle}
\endbibitem

\bibitem{Dar70}
\begin{barticle}[mr]
\bauthor{\bsnm{Darling},~\bfnm{D.~A.}\binits{D.~A.}}
(\byear{1970}).
\btitle{The {G}alton--{W}atson process with infinite mean}.
\bjournal{J. Appl. Probab.}
\bvolume{7}
\bpages{455--456}.
\bid{issn={0021-9002}, mr={0267648}}
\bptok{imsref}%
\end{barticle}
\endbibitem

\bibitem{DeHo91}
\begin{barticle}[mr]
\bauthor{\bsnm{Dekking},~\bfnm{F.~M.}\binits{F.~M.}} \AND
  \bauthor{\bsnm{Host},~\bfnm{B.}\binits{B.}}
(\byear{1991}).
\btitle{Limit distributions for minimal displacement of branching random
  walks}.
\bjournal{Probab. Theory Related Fields}
\bvolume{90}
\bpages{403--426}.
\bid{doi={10.1007/BF01193752}, issn={0178-8051}, mr={1133373}}
\bptok{imsref}%
\end{barticle}
\endbibitem

\bibitem{Do84}
\begin{barticle}[mr]
\bauthor{\bsnm{Doney},~\bfnm{R.~A.}\binits{R.~A.}}
(\byear{1984}).
\btitle{A note on some results of {S}chuh}.
\bjournal{J. Appl. Probab.}
\bvolume{21}
\bpages{192--196}.
\bid{issn={0021-9002}, mr={0732685}}
\bptok{imsref}%
\end{barticle}
\endbibitem


\bibitem{Feller}
\begin{bbook}[auto:STB|2013/01/29|08:09:18]
\bauthor{\bsnm{Feller},~\bfnm{W.}\binits{W.}}
(\byear{1970}).
\btitle{An Introduction to Probability Theory and Its Applications,
Vol. 2}.
\bpublisher{Wiley}, \blocation{New York}.
\bptok{imsref}%
\end{bbook}
\endbibitem

\bibitem{Gel67}
\begin{barticle}[mr]
\bauthor{\bsnm{Gel'fond},~\bfnm{A.~O.}\binits{A.~O.}}
(\byear{1967}).
\btitle{On a uniqueness theorem}.
\bjournal{Mat. Zametki}
\bvolume{1}
\bpages{321--324}.
\bid{issn={0025-567X}, mr={0208038}}
\bptok{imsref}%
\end{barticle}
\endbibitem

\bibitem{Gr89}
\begin{barticle}[mr]
\bauthor{\bsnm{Grey},~\bfnm{D.~R.}\binits{D.~R.}}
(\byear{1989}).
\btitle{A note on explosiveness of {M}arkov branching processes}.
\bjournal{Adv. in Appl. Probab.}
\bvolume{21}
\bpages{226--228}.
\bid{doi={10.2307/1427205}, issn={0001-8678}, mr={0980744}}
\bptok{imsref}%
\end{barticle}
\endbibitem

\bibitem{Ham74}
\begin{barticle}[mr]
\bauthor{\bsnm{Hammersley},~\bfnm{J.~M.}\binits{J.~M.}}
(\byear{1974}).
\btitle{Postulates for subadditive processes}.
\bjournal{Ann. Probab.}
\bvolume{2}
\bpages{652--680}.
\bid{mr={0370721}}
\bptok{imsref}%
\end{barticle}
\endbibitem

\bibitem{Har63}
\begin{bbook}[mr]
\bauthor{\bsnm{Harris},~\bfnm{Theodore~E.}\binits{T.~E.}}
(\byear{1963}).
\btitle{The Theory of Branching Processes}.
\bseries{Die Grundlehren der Mathematischen Wissenschaften}
\bvolume{119}.
\bpublisher{Springer}, \blocation{Berlin}.
\bid{mr={0163361}}
\bptok{imsref}%
\end{bbook}
\endbibitem

\bibitem{HuSh09}
\begin{barticle}[mr]
\bauthor{\bsnm{Hu},~\bfnm{Yueyun}\binits{Y.}} \AND
  \bauthor{\bsnm{Shi},~\bfnm{Zhan}\binits{Z.}}
(\byear{2009}).
\btitle{Minimal position and critical martingale convergence in branching
  random walks, and directed polymers on disordered trees}.
\bjournal{Ann. Probab.}
\bvolume{37}
\bpages{742--789}.
\bid{doi={10.1214/08-AOP419}, issn={0091-1798}, mr={2510023}}
\bptok{imsref}%
\end{barticle}
\endbibitem

\bibitem{Kal97}
\begin{bbook}[mr]
\bauthor{\bsnm{Kallenberg},~\bfnm{Olav}\binits{O.}}
(\byear{1997}).
\btitle{Foundations of Modern Probability}.
\bpublisher{Springer}, \blocation{New York}.
\bid{mr={1464694}}
\bptok{imsref}%
\end{bbook}
\endbibitem


\bibitem{Kin75}
\begin{barticle}[mr]
\bauthor{\bsnm{Kingman},~\bfnm{J.~F.~C.}\binits{J.~F.~C.}}
(\byear{1975}).
\btitle{The first birth problem for an age-dependent branching process}.
\bjournal{Ann. Probab.}
\bvolume{3}
\bpages{790--801}.
\bid{mr={0400438}}
\bptok{imsref}%
\end{barticle}
\endbibitem

\bibitem{Kol86}
\begin{bbook}[mr]
\bauthor{\bsnm{Kolchin},~\bfnm{Valentin~F.}\binits{V.~F.}}
(\byear{1986}).
\btitle{Random Mappings}.
\bpublisher{Optimization Software Inc. Publications Division}, \blocation{New
  York}.
\bid{mr={0865130}}
\bptok{imsref}%
\end{bbook}
\endbibitem

\bibitem{McD95}
\begin{barticle}[mr]
\bauthor{\bsnm{McDiarmid},~\bfnm{Colin}\binits{C.}}
(\byear{1995}).
\btitle{Minimal positions in a branching random walk}.
\bjournal{Ann. Appl. Probab.}
\bvolume{5}
\bpages{128--139}.
\bid{issn={1050-5164}, mr={1325045}}
\bptok{imsref}%
\end{barticle}
\endbibitem

\bibitem{Pet75}
\begin{bbook}[mr]
\bauthor{\bsnm{Petrov},~\bfnm{V.~V.}\binits{V.~V.}}
(\byear{1975}).
\btitle{Sums of Independent Random Variables}.
\bpublisher{Springer}, \blocation{New York}.
\bid{mr={0388499}}
\bptok{imsref}%
\end{bbook}
\endbibitem

\bibitem{Schu82}
\begin{barticle}[mr]
\bauthor{\bsnm{Schuh},~\bfnm{H.~J.}\binits{H.~J.}}
(\byear{1982}).
\btitle{Sums of i.i.d. random variables and an application to the explosion
  criterion for {M}arkov branching processes}.
\bjournal{J. Appl. Probab.}
\bvolume{19}
\bpages{29--38}.
\bid{issn={0021-9002}, mr={0644417}}
\bptok{imsref}%
\end{barticle}
\endbibitem

\bibitem{Sen69}
\begin{barticle}[mr]
\bauthor{\bsnm{Seneta},~\bfnm{E.}\binits{E.}}
(\byear{1969}).
\btitle{Functional equations and the {G}alton--{W}atson process}.
\bjournal{Adv. in Appl. Probab.}
\bvolume{1}
\bpages{1--42}.
\bid{issn={0001-8678}, mr={0248917}}
\bptok{imsref}%
\end{barticle}
\endbibitem

\bibitem{Sen73}
\begin{barticle}[mr]
\bauthor{\bsnm{Seneta},~\bfnm{E.}\binits{E.}}
(\byear{1973}).
\btitle{The simple branching process with infinite mean. {I}}.
\bjournal{J. Appl. Probab.}
\bvolume{10}
\bpages{206--212}.
\bid{issn={0021-9002}, mr={0413293}}
\bptok{imsref}%
\end{barticle}
\endbibitem

\bibitem{Sev67}
\begin{barticle}[mr]
\bauthor{\bsnm{Sevast'janov},~\bfnm{B.~A.}\binits{B.~A.}}
(\byear{1967}).
\btitle{Regularity of branching processes}.
\bjournal{Mat. Zametki}
\bvolume{1}
\bpages{53--62}.
\bid{issn={0025-567X}, mr={0216595}}
\bptok{imsref}%
\end{barticle}
\endbibitem

\bibitem{Sev70}
\begin{barticle}[auto:STB|2013/01/29|08:09:18]
\bauthor{\bsnm{Sevast'yanov},~\bfnm{B.~A.}\binits{B.~A.}}
(\byear{1970}).
\btitle{A necessary condition for the regularity of branching processes}.
\bjournal{Mat. Zametki}
\bvolume{7}
\bpages{389--396}.
\bptok{imsref}%
\end{barticle}
\endbibitem

\bibitem{Vat76}
\begin{barticle}[mr]
\bauthor{\bsnm{Vatutin},~\bfnm{V.~A.}\binits{V.~A.}}
(\byear{1976}).
\btitle{A condition for the regularity of the {B}ellman--{H}arris branching
  process}.
\bjournal{Dokl. Akad. Nauk SSSR}
\bvolume{230}
\bpages{15--18}.
\bid{issn={0002-3264}, mr={0420897}}
\bptok{imsref}%
\end{barticle}
\endbibitem

\bibitem{Vat87}
\begin{barticle}[mr]
\bauthor{\bsnm{Vatutin},~\bfnm{V.~A.}\binits{V.~A.}}
(\byear{1987}).
\btitle{Sufficient conditions for the regularity of
Bellman--Harris branching processes}.
\bjournal{Theory Probab. Appl.}
\bvolume{31}
\bpages{50--57}.
\bptok{imsref}%
\end{barticle}
\endbibitem

\bibitem{Vat96}
\begin{barticle}[mr]
\bauthor{\bsnm{Vatutin},~\bfnm{V.~A.}\binits{V.~A.}}
(\byear{1996}).
\btitle{On the explosiveness of
nonhomogeneous age-dependent
branching processes}.
\bjournal{Theory Probab. Math. Statist.}
\bvolume{52}
\bpages{39--42}.
\bnote{Translated from \textit{Teor. Imovir. Mat. Stat.} \textbf{52} (1995) 37--40
(Ukrainian)}.
\bptok{imsref}%
\end{barticle}
\endbibitem

\bibitem{VZ93}
\begin{barticle}[mr]
\bauthor{\bsnm{Vatutin},~\bfnm{V.~A.}\binits{V.~A.}} \AND
  \bauthor{\bsnm{Zubkov},~\bfnm{A.~M.}\binits{A.~M.}}
(\byear{1993}).
\btitle{Branching processes. {II}. Probability theory and mathematical statistics, 1}.
\bjournal{J. Soviet Math.}
\bvolume{67}
\bpages{3407--3485}.
\bid{doi={10.1007/BF01096272}, issn={0090-4104}, mr={1260986}}
\bptok{imsref}%
\end{barticle}
\endbibitem

\end{thebibliography}
\end{document}